\documentclass{imanum}

\usepackage{newtxtext, newtxmath}  
\usepackage{fix-cm}        
\usepackage{anyfontsize}  

\usepackage{mathptmx}  
\usepackage{newtxmath}  

\usepackage{amsmath}
\usepackage{amssymb}
\usepackage{amsthm}

\usepackage{graphicx}
\usepackage{float}

\makeatletter
\def\endfigure{\end@float}  
\makeatother
\usepackage{subfigure}       

\usepackage{comment}
\usepackage{xcolor}  

\usepackage{natbib}

\usepackage{hyperref}

\usepackage{cleveref}

\graphicspath{{Fig/}}

\theoremstyle{plain}  

\usepackage{mathrsfs}

\jno{drnxxx}
\received{Date Month Year}
\revised{Date Month Year}

\def\x{\boldsymbol{x}}

\def\v{\mathbf{v}}
\def\V{\mathbf{V}}

\def\e{\mathbf{e}}
\def\f{\mathbf{f}}

\def\h{\mathbf{h}}

\def\H{\mathbb{H}}
\def\m{\mathbf{m}}
\def\M{\mathbf{M}}
\def\n{\mathbf{n}}

\def\d{\mathrm{d}}
\def\eff{\mathrm{eff}}

\def\eps{\epsilon}

\def\0{\mathbf{0}}

\def\S{\mathbf{S}}

\def\R{\mathbf{R}}

\def\<{\langle}
\def\>{\rangle}
\newcommand{\LOD}{\mathrm{LOD}}

\definecolor{lzcol}{rgb}{1, 0, 0}

\begin{document}


\title{Optimal Error Estimates of a Linearized Backward Euler Localized Orthogonal Decomposition for the Landau-Lifshitz Equation }

\shorttitle{Optimal Error Estimates of a Linearized Backward Euler LOD for the LL Equation }

\author{%
{\sc
Zetao Ma} \\[2pt]
School of Mathematical Sciences, Shanghai Jiao Tong University, Shanghai 200240, P.R. China\\[6pt]
{\sc
Rui Du\thanks{Corresponding author. Email: durui@suda.edu.cn}} \\[2pt]
School of Mathematical Sciences and Mathematical Center for Interdisciplinary Research, Soochow University, Suzhou 215006, P.R. China\\[6pt]
{\sc and}\\[6pt]
{\sc Lei Zhang}\\[2pt]
School of Mathematical Sciences, MOE-LSC and Institute of Natural Sciences, Shanghai Jiao Tong University, Shanghai 200240, P.R. China
}

\shortauthorlist{Z. T. Ma \emph{et al.}}

\maketitle

\begin{abstract}
{We introduce a novel spatial discretization technique for the reliable and efficient simulation of magnetization dynamics governed by the Landau-Lifshitz (LL) equation. The overall discretization error is systematically decomposed into temporal and spatial components. The spatial error analysis is conducted by formulating the LL equation within the framework of the Localized Orthogonal Decomposition (LOD) method. Numerical examples are presented to validate the accuracy and approximation properties of the proposed scheme.}
{finite element method; localized orthogonal decomposition; two-level multiscale methods; computational micromagnetism; Landau-Lifshitz equation.}
\end{abstract}

\section{Introduction}
\label{sec1}
The magnetization dynamics in a ferromagnetic material is governed by the Landau-Lifshitz (LL) equation \cite{landau1992theory, gilbert1955lagrangian}, which is crucial for understanding and predicting the behavior of magnetic materials, forming the theoretical basis for modern magnetism studies and applications in spintronics, magnetic memory devices, and other advanced technologies\cite{gutfleisch2011magnetic}. Denoting the magnetization vector by \(\mathbf{m}: \Omega \times (0,T] \rightarrow \mathbb{S}^2\), where \(\mathbb{S}^2\) is the unit sphere and \(\Omega \subset \mathbb{R}^d\) (with \(d=2,3\)) is a bounded, smooth, and convex domain. The dimensionless form of the LL equation is given by:
\begin{subequations}\label{LLeq}
    \begin{align}
        \partial_t \mathbf{m} &= -\mathbf{m} \times \mathbf{h}_{\text{eff}} - \alpha \mathbf{m} \times (\mathbf{m} \times \mathbf{h}_{\text{eff}}), &&\mathbf{x} \in \Omega, t \in (0,T], \\
        \kappa(\mathbf{x}) \frac{\partial \mathbf{m}}{\partial \mathbf{n}} &= \mathbf{0}, &&\mathbf{x} \in \partial \Omega, t \in (0,T], \\
        \mathbf{m}(\mathbf{x}, 0) &= \mathbf{m}_0(\mathbf{x}), &&\mathbf{x} \in \Omega,
    \end{align}
\end{subequations}
where \(\alpha > 0\) is the dimensionless Gilbert damping constant \cite{gilbert1955lagrangian}, \(\mathbf{n}\) is the unit outward normal vector, and the initial data \(\mathbf{m}_0\) satisfies \(|\mathbf{m}_0| = 1\). The effective field \(\mathbf{h}_{\text{eff}}\) is the focus of this work.

Following the framework of many analytical studies, we consider the exchange interaction as the dominant contribution to the effective field. Specifically, we define:
\begin{equation}\label{effectivefield}
    \mathbf{h}_{\text{eff}}[\mathbf{m}] = \nabla \cdot (\kappa \nabla \mathbf{m}),
\end{equation}
where \(\kappa(\mathbf{x})\) represents rough coefficients in \(L^{\infty}(\Omega)\), which are non-periodic and exhibit non-separable scales. The coefficient \(\kappa\) satisfies the uniform ellipticity condition:
\begin{equation*}
    \kappa_{\text{min}} |\xi|^2 \le \xi^\top \kappa(\mathbf{x}) \xi \le \kappa_{\text{max}} |\xi|^2, \quad \forall \xi \in \mathbb{R}^d, \text{ a.e. } \mathbf{x} \in \Omega.
\end{equation*}
The corresponding Landau-Lifshitz energy functional is:
\begin{equation}\label{LLenergy}
    F[\mathbf{m}] = \frac{1}{2} \int_{\Omega} \kappa |\nabla \mathbf{m}|^2 \, d\mathbf{x} = -\frac{1}{2} \left( \mathbf{h}_{\text{eff}}[\mathbf{m}], \mathbf{m} \right).
\end{equation}

Accurately modeling spin-wave dynamics in disordered magnetic solids—crucial for advancing both fundamental wave physics in complex media and next-generation, low-power technologies \cite{liu1979spin,gutfleisch2011magnetic,alouges2021stochastic}—requires treating the exchange term with rough coefficients. This is a general feature of ferromagnets that does not assume ergodicity or scale separation. Existing studies of the Landau-Lifshitz equation with multiscale features offer a valuable foundation \cite{santugini2007homogenization,alouges2015homogenization,choquet2018homogenization,alouges2021stochastic,leitenmaier2022homogenization,chen2022multiscale,UpscalingHMM,leitenmaier2023finite}, yet the specific challenge of rough coefficients without structural assumptions remains unresolved. This work is motivated by the need to close this gap.

Obtaining high-resolution solutions for multiscale problems—such as those with rough coefficients, periodic structures, or scale-separated parameterizations such as \(a(\mathbf{x}/\varepsilon)\) (where \(\varepsilon \ll 1\) as in \cite{leitenmaier2022heterogeneous})—is computationally expensive due to the fine mesh required to resolve \(\varepsilon\)-scale features. Moreover, assumptions of periodicity or scale separation often do not hold for real physical materials \cite{liu1979spin,gutfleisch2011magnetic}. Numerical homogenization methods have been developed to address these limitations in cost and generality. These approaches solve multiscale PDEs by constructing localized basis functions on a coarse mesh, preserving the sparsity and complexity of standard FEM while ensuring accuracy with respect to the coarse mesh size. Typical methods include heterogeneous multiscale methods \cite{ee03,abdulle2014analysis,ming2005analysis}, numerical homogenization \cite{dur91,ab05,weh02}, multiscale finite element methods \cite{Arbogast_two_scale_04,egw10,eh09}, the multiscale spectral generalized finite element method \cite{babuska2011optimal,efendiev2013generalized,ma2022novel}, the localized orthogonal decomposition (LOD) \cite{maalqvist2014localization,henning2013oversampling,kornhuber2018analysis,hauck2023super}, flux norm homogenization \cite{berlyand2010flux,Owhadi:2011}, (generalized) rough polyharmonic splines \cite{zhang_RPS,liuGRPS}, and gamblets \cite{owhadi2017gamblets,owhadi2017multigrid}. Comprehensive overviews are available in  \cite{book_Peterseim,book_owhadi2019operator,chung2023multiscale,owhadi2017multigrid,altmann2021numerical}.

The main result of this paper is an error analysis between the exact solution \(\mathbf{m}\) and the numerical homogenization solution \(\mathbf{M}_{\text{LOD}}\). We decompose the total error into independent temporal and spatial components—a departure from standard finite element error analysis. To bound specific terms in the spatial error analysis, our approach relies on a higher-regularity temporal error estimate, employing a strategy analogous to \cite{henning2022superconvergence}. The spatial error is then further split into the LOD projection error and the discrepancy between the LOD projection of the temporally discretized solution and \(\mathbf{M}_{\text{LOD}}\). In this work, we set \(\kappa = 1\) (so that \(\mathbf{h}_{\text{eff}} = \Delta \mathbf{m}\)). This simplification serves a dual purpose: (i) it circumvents the technical difficulties and stricter regularity conditions that arise from applying differential operators—especially of higher order—to variable coefficients, and (ii) it aligns with the standard analytical assumption \(\kappa \in C^{\infty}(\Omega)\) for the multiscale Landau-Lifshitz equation \cite{leitenmaier2022homogenization, UpscalingHMM}, which inherently governs the behavior of all derivatives of \(\kappa\). This approach, common in numerical homogenization for nonlinear PDEs (e.g., within the LOD framework \cite{henning2022superconvergence, henning2023optimal}), thereby allows for a simpler exposition.

The novel theoretical contributions include a convergence proof for the Landau-Lifshitz equation in the LOD space, demonstrating an \(L^{\infty}(0,T;L^2(\Omega))\)-convergence rate of order \(O(\tau + H^3)\), where \(\tau\) is the time step size and \(H\) the coarse mesh size. This represents a significant improvement over standard finite element methods, which achieve only \(O(\tau + h^2)\) convergence in the \(L^2\)-norm \cite{gao2014optimal,anrongbackward}. Unlike standard FEM approaches for the Landau-Lifshitz equation \cite{li2012error,li2013unconditional,gao2014optimal,an2016LLCN}, our methodology is based on a combined framework of temporal-spatial decoupling and spatial error decomposition, along with a novel two-step error splitting. This splitting critically relies on establishing a temporal error estimate with higher regularity to control key spatial terms—a technique that diverges from conventional analysis and follows the strategy in \cite{henning2022superconvergence}.

\paragraph{Organization:} The paper is structured as follows. Section \ref{sec: Localized Orthogonal Decomposition} reviews the Localized Orthogonal Decomposition (LOD) approach for elliptic problems, including its approximation and localization properties. Section \ref{sec: Analysis of backward Euler discretization of LL equation in LOD space} details the temporal and spatial error analysis of the Landau-Lifshitz (LL) equation, with particular emphasis on the spatial estimate within the LOD framework. Section \ref{sec: Numerical results} presents numerical examples to validate the theoretical estimates. The detailed proofs of the main theorems are provided in Section \ref{sec11}, and the Appendix summarizes key technical inequalities and the variants of the backward Euler schemes.

\paragraph{Notation:} We write $\mathbb{H}^1(\Omega) = [H^1(\Omega)]^3$ and $\mathbb{L}^2(\Omega) = [L^2(\Omega)]^3$ for the vector-valued Sobolev and Lebesgue spaces, respectively.

\section{An Ideal Multiscale Method: Localized Orthogonal Decomposition}
\label{sec: Localized Orthogonal Decomposition}

\subsection{Ideal Multiscale Method and Approximation Properties}
\label{subsec: Ideal multiscale method and approximation properties}
Throughout this section, the domain $\Omega \subseteq \mathbb{R}^d$ ($d=2,3$) is assumed to be bounded, smooth, and convex. We first provide a concise overview of the LOD, recalling several approximation results that are fundamental to our later error estimates. A thorough treatment of these results, including proofs and applications in low-regularity contexts, is available in \cite{maalqvist2014localization, hauck2023super, zhang_RPS, liuGRPS}.

Let \( P_H: L^2(\Omega) \to V_H \) denote the \( L^2 \)-orthogonal projection onto the finite-dimensional subspace \( V_H \subset H^1(\Omega) \). The kernel of \( P_H \), defined by $W := \ker(P_H) = \{ w \in H^1(\Omega) \mid P_H w = 0 \}$, 
which is a closed subspace of \( H^1(\Omega) \). This induces the \( L^2 \)-orthogonal decomposition $H^1(\Omega) = V_H \oplus W$.

Let \( (\cdot,\cdot) \) denote the \( L^2 \)-inner product, and let \( A(\cdot, \cdot) \) be a symmetric, coercive bilinear form defining an inner product on \( H^1(\Omega) \). With respect to the \( A \)-inner product, \( H^1(\Omega) \) admits a second orthogonal decomposition:
\[
H^1(\Omega) = V_{\LOD} \oplus_A W,
\]
where the \textbf{localized orthogonal decomposition (LOD)} space \( V_{\LOD} \) is defined as the \( A \)-orthogonal complement of \( W \):
\begin{equation}\label{def_LOD_space}
    V_{\LOD} = \{ v \in H^1(\Omega) \mid A(v, w) = 0 \quad \forall w \in W \}.
\end{equation}
This construction yields a low-dimensional multiscale space \( V_{\LOD} \) with \( \dim(V_{\LOD}) = \dim(V_H) \), which is well-suited for numerical approximation of multiscale problems. Let \( u_{\LOD} \in V_{\LOD} \) be the Ritz projection of $u$  with 
\begin{equation}\label{ritz_LOD}
A(u_{\mathrm{LOD}}, v) = A(u, v) \qquad \forall v \in V_{\mathrm{LOD}}.
\end{equation}
This is particularly evident when expressed in the standard notation
\begin{equation}\label{eq:LOD-space-def}
V_{\LOD} = (\mathrm{Id} - Q)V_H,
\end{equation}
where the operator $Q: V_H \to W$ is commonly referred to as the correction operator.

The following result quantifies the global approximation quality of the LOD method over the whole domain $\Omega$ for elliptic problems:
\begin{theorem}[Global approximation error]\label{theorem_convergence_rate}
Let \( u \) be the solution to the elliptic problem with right-hand side \( f \in H^s(\Omega) \), \( s = 0, 1 \), and let \( u_{\text{LOD}} \in V_{\text{LOD}} \) be defined by \eqref{ritz_LOD}. Then,
 \[
 \| u - u_{\text{LOD}} \|_{L^2(\Omega)} + H \| u - u_{\text{LOD}} \|_{H^1(\Omega)} \le C H^{s+1} \| f \|_{H^s(\Omega)},
 \]
where the constant \( C > 0 \) is independent of \( H \), \( u \), and the regularity of \( u \).
\end{theorem}

\subsection{Localization Property} 
\label{subsec: Localization properties}
This subsection presents the localization properties of the global problem \eqref{def_LOD_space}, where the set of global basis functions is defined by \eqref{eq:LOD-space-def}. Let $\mathcal{T}_H$ denote a coarse mesh of $\Omega$ with characteristic size $H$. For any $K \in \mathcal{T}_H$, the $\ell$-layer patch $\Omega^\ell(K)$ is defined recursively: $\Omega^0(K) := K$ and for $\ell \geq 1$,
\begin{equation*}
    \Omega^\ell(K) := \bigcup_{\substack{K \in \mathcal{T}_H \\ \overline{K} \cap \overline{\Omega^{\ell-1}(K)} \neq \varnothing}} K .
\end{equation*}
Denote by $W(\Omega^\ell(K))$ the restriction of the fine-scale space $W$ to the patch $\Omega^\ell(K)$.
Then, we can define the \textbf{localized} bases over $\Omega^\ell(K)$. For each $v_H \in V_H$ and $K \in \mathcal{T}_H$, the local correction operator $Q_K^{\ell}(v_H) \in W(\Omega^\ell(K))$ is defined as
\begin{equation}
   A(Q_K^{\ell}(v_H), w^{\ell}) =- A_K(v_H, w^{\ell}) \qquad \forall w^{\ell} \in W(\Omega^\ell(K)),
\end{equation}
where $ A_K(\cdot, \cdot)$ represents the restriction of $A(\cdot, \cdot)$ on some element $K$. With this, the correction function is given by $R_\ell(v_H) := v_H + \sum \limits _{K \in \mathcal{T}_H} Q_{K}^{\ell}(v_H).$ Finally, the \textbf{localized} space over $\Omega^\ell(K)$ is defined by
\begin{equation}\label{local_LOD}
    V_{\LOD}^{\ell}:=\{R_\ell(v_H) \mid  v_H \in V_H \}.
\end{equation}
The following result quantifies the localization error between the global LOD solution $u_{\LOD}$ defined by \eqref{ritz_LOD} and its localized counterpart on patch $\Omega^\ell(K)$:
\begin{theorem}[Localization error]\label{theorem_localization}
Let \( u_{\text{LOD}} \in V_{\text{LOD}} \) be defined by \eqref{ritz_LOD}, and let \( u^{\ell}_{\text{LOD}} \) denote its Ritz projection in \( V_{\text{LOD}}^{\ell} \).
Then,
\[
\| u_{\LOD} - u^{\ell}_{\LOD} \|  \le C \exp(-\rho \ell) \|f \|_{H^1},
\]
where $C > 0$ is as in Theorem \ref{theorem_convergence_rate}, and $\rho > 0$ depends only on the contrast $\kappa_{\max}/\kappa_{\min}$ but not on $h$, $H$, or the oscillations in $\kappa$.  
\end{theorem}

\begin{remark}
This estimate reveals the exponential decay of the localization error with respect to the oversampling parameter \(\ell\), a key characteristic of the LOD method. The construction of these localized basis functions is detailed in \cite{maalqvist2014localization, hauck2023super, zhang_RPS, liuGRPS}. With \( \ell \geq 3 \log(H^{-1}) / \rho \), the optimal convergence rates of \( O(H^3) \) in the \( L^2 \) norm and \( O(H^2) \) in the \( H^1 \) norm are preserved, provided that \( f \in H^1(\Omega) \).
\end{remark}

\begin{remark}
The localization property of LOD can be improved, as demonstrated by the Super-localized Orthogonal Decomposition (SLOD) \cite{hauck2023super}. SLOD has been extended to reaction-convection-diffusion equations \cite{bonizzoni2024reduced} and convection-dominated diffusion problems \cite{bonizzoni2022super}.
\end{remark}

\begin{theorem}[Approximation error with localization]\label{theorem_total}
Let \( u \) be the solution to the elliptic problem, and let \( u^{\ell}_{\text{LOD}} \) denote its Ritz projection in \( V_{\text{LOD}}^{\ell} \). Then,
 \[
 \| u - u^{\ell}_{\text{LOD}} \|_{L^2} \le C \big( H^3 + \exp(-\rho \ell) \big) \| f \|_{H^1},
 \]
 \[
 \| u - u^{\ell}_{\text{LOD}} \|_{H^1} \le C \big( H^2 + \exp(-\rho \ell) \big) \| f \|_{H^1},
 \]
 where the constants \( C \) and \( \rho \) are as in Theorem \ref{theorem_localization}.
\end{theorem}

\section{Analysis of Backward Euler Discretization of LL equation in LOD Space}
\label{sec: Analysis of backward Euler discretization of LL equation in LOD space}

This section outlines the analytical framework for error estimation, which is decomposed into temporal and spatial components. To establish optimal convergence rates, the following regularity assumptions are required.

\begin{assumption}\label{assumption} 
    We assume the following conditions hold:
    \begin{enumerate}
        \item The solution possesses sufficient regularity. Specifically, in three-dimensional domains, we assume:
        \begin{equation}\label{ass:regularity_estimate}
        \|\mathbf{m}_0\|_{H^4(\Omega)} + \| \partial_t \mathbf{m} \|_{L^{\infty}(0,T;H^4(\Omega))} 
        + \| \partial_{tt} \mathbf{m} \|_{L^{2}(0,T;H^2(\Omega))} 
        + \|\mathbf{m} \|_{L^{\infty}(0,T;W^{2,4}(\Omega))} \le C_{\mathrm{reg}},
        \end{equation}
        where $C_\mathrm{reg}$ is a generic constant. We note that in two-dimensional settings, 
        the assumption $\|\mathbf{m} \|_{L^{\infty}(0,T;W^{2,4}(\Omega))}$ can be relaxed to 
        $\|\mathbf{m} \|_{L^{\infty}(0,T;W^{2,2+\delta}(\Omega))}$ for some $\delta > 0$ \cite{gao2014optimal}.
        \item The domain $\Omega \subset \mathbb{R}^d$ is of class $C^{2,1}$.
    \end{enumerate}
\end{assumption}

\begin{remark} The assumption is motivated by the following considerations.
\begin{enumerate}
\item The regularity assumptions on \( \mathbf{m}_0 \) and \( \partial_t \mathbf{m} \) are standard. Similar prerequisites appear in \cite{henning2022superconvergence} for the Gross–Pitaevskii equation (Assumption (A6), Lemma 10.4). Furthermore, the need for high initial regularity is highlighted by \cite{an2016LLCN}, who require \( \mathbf{m}_0 \in H^5(\Omega) \) to secure \( L^{\infty}(0,T;H^3(\Omega)) \) regularity for the LL equation. Our assumption \( \mathbf{m}_0 \in H^4(\Omega) \) is thus consistent with, and weaker than these prior studies.
\item \label{truncation}
The regularity condition \( \partial_{tt} \mathbf{m} \in L^{2}(0,T;H^2(\Omega)) \) ensures that the temporal truncation error satisfies the bound \( \tau \sum_{k=0}^{n} \| \mathbf{R}_{\mathrm{tr}}^{k} \|_{H^2(\Omega)}^2 \le C \tau^2 \). If only the weaker condition \( \partial_{tt} \mathbf{m} \in L^{2}(0,T;L^2(\Omega)) \) holds, the corresponding estimate in the \( L^2 \)-norm remains valid.
\item The boundedness of \(\|\mathbf{m}\|_{L^{\infty}(0,T;W^{2,4}(\Omega))}\) is a common assumption in the analysis of the Landau–Lifshitz equation. While the boundedness of \(\|\nabla \mathbf{m}\|_{L^{2}(0,T;L^{2}(\Omega))}\) follows directly from the initial LL energy \cite{choquet2018homogenization, chen2022multiscale}, the boundedness of \(\|\nabla \mathbf{m}\|_{L^{\infty}(0,T;L^{\infty}(\Omega))}\) is not straightforward. Numerical evidence from \cite{leitenmaier2022homogenization, UpscalingHMM, leitenmaier_Phdthesis} suggests it may hold in certain regimes. To rigorously ensure this, we impose the \(W^{2,4}\)-regularity assumption and apply the Gagliardo–Nirenberg inequality in Appendix \ref{lem:GN_inequality}, which gives 
$\|\nabla \mathbf{m}\|_{L^{\infty}(0,T;L^{\infty}(\Omega))} \le C \|\mathbf{m}\|_{L^{\infty}(0,T;W^{2,4}(\Omega))} \le C.$
    \end{enumerate}
\end{remark}

\subsection{Discretization Scheme}

The LL equation \eqref{LLeq} can be reformulated using the unit-length constraint \(|\mathbf{m}| = 1\),
\begin{equation}\label{LLform1}
	\mathbf{m}_t - \alpha \Delta \mathbf{m} + \mathbf{m} \times \Delta \mathbf{m} = \alpha |\nabla \mathbf{m}|^2 \mathbf{m}.
\end{equation}

We discretize the time interval \([0,T]\) into \(N\) steps of size \(\tau := T/N\), with discrete times \(t_n := n\tau\) for \(n = 0,\dots,N\). Over the past two decades, several semi-implicit backward Euler schemes have been proposed for the time discretization of the LL equation \cite{cimrak2005error,gao2014optimal,anrongbackward}, differing mainly in their treatment of the nonlinear term. 

In this work, we adopt the scheme from \cite{cimrak2005error} to seek $\m_h\in\V_h$:
\begin{equation}\label{Cimrak}
    (D_{\tau} \m_h^{n+1},\v_h) + \alpha(\nabla \m_h^{n+1},\nabla \v_h)  
    - (\m_{h}^n \times \m_h^{n+1}, \v_h) 
    = \alpha\bigl((\nabla \m_h^{n} \cdot \nabla \m_h^n)\m_h^{n+1}, \v_h\bigr),
    \qquad \forall \v_h \in \V_h,
\end{equation}
where the discrete time derivative is defined as 
$\displaystyle D_{\tau} \varphi^{n+1}:=\frac{\varphi^{n+1} - \varphi^{n}}{\tau}$,
and $\V_h$ denotes a finite element space for vector-valued functions, such as the standard $P_1$ space. 

While Cimrák's analysis focuses primarily on temporal errors, it provides limited discussion of spatial discretization errors. To address this gap, we draw inspiration from the fully discrete finite element frameworks in \cite{gao2014optimal} and \cite{anrongbackward} (detailed in Appendix \ref{subsec: Some Other Backward Euler Schemes}), which reformulate the nonlinear term in \eqref{Cimrak} to achieve optimal error estimates. Building on these approaches, we conduct a rigorous, fully discrete error analysis within the LOD framework, with particular emphasis on spatial discretization exploiting approximation properties of $V_{\LOD}$.

Following the temporal-spatial error splitting argument \cite{li2012error,li2013unconditional}, we first analyze a semi-discrete scheme in time. Starting from $\M^0 := \m_0 \in \H^1(\Omega)$, we sequentially find $\M^{n+1} \in \H^1(\Omega)$ for $n = 0, \dots, N-1$ by solving the elliptic system
\begin{equation}\label{time_discrete_form}
 	\begin{aligned}
 		D_{\tau}\M^{n+1}  - \alpha  \Delta \M^{n+1}  + \M^{n} \times \Delta  \M^{n+1} +  \nabla \M^{n} \times \nabla \M^{n+1}
 		=\alpha (\nabla  \M^{n} \cdot \nabla \M^{n} )\M^{n+1},
 	\end{aligned}
\end{equation}
with the Neumann condition $\nabla \M^{n+1} \cdot \n = \0$ and initial data $\M^0 = \m_0$.

The corresponding Galerkin formulation is obtained by multiplying \eqref{time_discrete_form} with a test function \(\v \in \H^1(\Omega)\) and integrating by parts. Specifically, find \(\M^{n+1}\in \H^1(\Omega)\) such that
\begin{equation}\label{time_discrete_elliptic}
\begin{aligned}
(D_{\tau}\M^{n+1},\v)+ \alpha(\nabla \M^{n+1},\nabla \v)
- ( \M^{n} \times \nabla \M^{n+1} , \nabla \v)
  = \alpha\big( (\nabla \M^{n}\cdot \nabla \M^{n})\M^{n+1}, \v\big),
  \end{aligned}
  \end{equation}
  for all \(\v \in \H^1(\Omega)\).

Combining the LOD spatial discretization with the time-stepping scheme in \eqref{Cimrak} yields the following fully discrete problem. Define the vector-valued LOD space as $\V_{\LOD} = [V_{\LOD}]^3$. Given \(\M_{\LOD}^0 \in \V_{\LOD}\), find \(\M_{\LOD}^{n+1}\in \V_{\LOD}\) for \(n=0,\ldots,N-1\) such that
\begin{equation}\label{nospeed}
\begin{aligned}
(D_{\tau}\M_{\LOD}^{n+1},\v) +
 \alpha(\nabla \M_{\LOD}^{n+1},\nabla \v)
- (\M_{\LOD}^{n} \times \nabla \M_{\LOD}^{n+1}, \nabla \v)
  = \alpha\big(|\nabla \M_{\LOD}^{n}|^2 \M_{\LOD}^{n+1}, \v\big),
  \end{aligned}
\end{equation}
for all \(\v \in \V_{\LOD}\).
We now state the main theorem of this work, which provides a fully discrete error estimate for the combined temporal and spatial discretizations.

\begin{theorem}\label{mainresult}
Let \( T > 0 \) be given, and suppose the Landau–Lifshitz equation \eqref{LLeq} admits a unique solution satisfying assumption \eqref{assumption}. For the elliptic system \eqref{time_discrete_form} with homogeneous Neumann boundary conditions, there exist constants \( \tau_1 > 0 \) and \( H_1 > 0 \) such that if \( \tau \le \tau_1 \) and \( H \le H_1 \), then:
\begin{enumerate}
    \item The system has a unique solution \( \mathbf{M}^{n+1} \).
    \item The LOD approximation \( A_{\LOD}(\mathbf{M}^{n+1}) \in \V_{\LOD} \) satisfies the LOD formulation \eqref{bilinear_B}.
\end{enumerate}
Moreover, the following error estimates hold:
\begin{align}
    \|\m(t_n) - \M_{\LOD}^n \|_{L^2(\Omega)} &\le C\,(\tau + H^3), \label{eq:L2-error} \\
    \|\m(t_n) - \M_{\LOD}^n \|_{H^1(\Omega)} &\le C\,(\tau + H^2), \label{eq:H1-error}
\end{align}
where the constant $C > 0$ is independent of the mesh size $H$ and the time step $\tau$, but may depend on $\Omega$, $\alpha$, and the regularity constant $C_\mathrm{reg}$ in Assumption \ref{assumption}. 
\end{theorem}

The remainder of this section is organized as follows. In Section~\ref{subsec:Temporal Error Estimate} we present the temporal error estimate, and in Section~\ref{subsec:Spatial Estimate} we establish the spatial error estimate. Finally, in Section~\ref{subsec:Proof of Main Theorem} we combine these results to complete the proof of Theorem~\ref{mainresult}.

\subsection{Temporal Error Estimate}
\label{subsec:Temporal Error Estimate}
Next, we analyze the temporal error $\e^n := \M^n - \m^n$, where $\m^n := \m(t_n)$ denotes the exact magnetization at time $t_n$. Subtracting the weak form of the continuous equation \eqref{LLform1} from the discrete scheme \eqref{time_discrete_form} yields the following error equation,
\begin{equation}\label{errorform}
	\begin{aligned}
		& D_{\tau}\e^{n+1} - \alpha  \Delta \e^{n+1}  + \M^{n} \times \Delta  \e^{n+1} +   \nabla \M^{n} \times \nabla \e^{n+1} + \e^{n} \times \Delta  \m^{n+1} + \nabla \e^{n} \times \nabla \m^{n+1}\\
		=&\alpha ( |\nabla  \M^{n} |^2 \M^{n+1} -|\nabla  \m^{n} |^2 \m^{n+1} ) -\R_{\text{tr}}^{n+1},
	\end{aligned}
\end{equation}
with $\tau \sum \limits _{k=0} ^{n} ||\R_{\text{tr}}^{n}||_{L^2}^2 \le \tau^2$, $\tau \sum \limits _{k=0} ^{n} ||\Delta \R_{\text{tr}}^{n}||_{L^2}^2 \le \tau^2$, where the truncation error $\R_{\text{tr}}$ is given by
\begin{equation}\label{R_tr}
	\begin{aligned}
		\R_{\text{tr}}^{n+1}
		=D_{\tau} \m^{n+1} - \m_t^{n+1} - (\m^{n+1} -\m^{n}) \times \nabla \m^{n+1} + \alpha (|\nabla \m^{n+1}|^2\m^{n+1} - |\nabla \m^{n}|^2\m^{n+1}).
	\end{aligned}
\end{equation}
Hence, the weak formulation of the temporal error equation can be expressed as follows: 
\begin{equation*}\label{weak_time}
	\begin{aligned}
		& (D_{\tau} \e^{n+1},\v ) + \alpha  (\nabla \e^{n+1}, \nabla \v)  - ( \e^{n} \times \nabla  \m^{n+1}, \nabla \v) -  ( \M^{n} \times \nabla \e^{n+1},\nabla \v ) \\
		=&\alpha ( |\nabla  \M^{n} |^2 \M^{n+1} -|\nabla  \m^{n} |^2 \m^{n+1},\v ) -(\R_{\text{tr}}^{n+1}, \v ),
	\end{aligned}
\end{equation*}
for all $\v \in \H^1( \Omega )$, where $\e^0 = \0$ in $\Omega$ and $\nabla \e^n \cdot \n = \0$ on $\partial \Omega .$

We present the theorem establishing the temporal error bound.
\begin{theorem}\label{temporal_1}
Let $T>0$ represent a specified constant, and assume that the LL equation \eqref{LLeq} admits a unique solution that fulfills assumption \eqref{assumption}. Then, the elliptic system \eqref{time_discrete_form} with homogeneous Neumann boundary condition admits a unique solution $\M^{j+1}$ such that $\tau \le \tau_1$ for some $ \tau_1 >0$ ,
    \begin{equation}\label{time_discre}
    \sup_{0 \le k \le n}( ||\e^k||_{H^1}^2 + \tau \sum \limits _{i=0} ^{k} ||\e^i||_{H^2}^2 ) \le C_0^2 \tau^2 ,
    \end{equation}
where $C_0$ is a constant independent of the time step $\tau$, but may depend on $\Omega$, $\alpha$, and the regularity constant $C_\mathrm{reg}$ in Assumption \ref{assumption}.
\end{theorem}

The detailed proof is provided in Section \ref{subsec:Proof of Theorem ref{temporal_1}.}, where we derive a temporal error estimate for \( \| e^{n+1} \|_{H^1} \), which directly characterizes the scheme's accuracy. The remaining two estimates, \( \| \Delta e^{n+1} \|_{L^2} \) and \( \| \nabla \Delta e^{n+1} \|_{L^2} \), are established later in Section \ref{subsec_lemma:Auxiliary_Lemmas}. These higher-order estimates, commonly employed in LOD analysis for nonlinear problems \cite{henning2022superconvergence}, are crucial as they guarantee the \( H^1 \) regularity of the forcing term in the associated elliptic system, which in turn enables the \( H^1 \) error analysis.

\subsection{Spatial Estimate}
\label{subsec:Spatial Estimate}

In this section, we analyze the spatial error \( \| \mathbf{M}^n - \mathbf{M}_{\text{LOD}}^n \| \). Since this error cannot be estimated directly, we decompose it into two components: the projection error \( \| \mathbf{M}^n - A_{\text{LOD}}(\mathbf{M}^n) \| \) and the approximation error within the LOD space \( \| A_{\text{LOD}}(\mathbf{M}^n) - \mathbf{M}_{\text{LOD}}^n \| \).

For \( n = 0, 1, \dots, N-1 \), the bilinear form \( B^n(\mathbf{u}, \mathbf{v}) \) is defined as follows:

\[
B^0(\mathbf{u}, \mathbf{v}) = \alpha (\nabla \mathbf{u}, \nabla \mathbf{v}) -  (\mathbf{M}^0 \times \nabla \mathbf{u}, \nabla \mathbf{v}) + \alpha (\mathbf{u}, \mathbf{v}),
\]

and for \( n > 0 \),

\begin{equation}
    B^{n+1}(\mathbf{u}, \mathbf{v}) = \alpha (\nabla \mathbf{u}, \nabla \mathbf{v}) -  (\mathbf{M}^n \times \nabla \mathbf{u}, \nabla \mathbf{v}) + \alpha (\mathbf{u}, \mathbf{v}).
\label{bilinear_B}
\end{equation}
where the term \( (\mathbf{u}, \mathbf{v}) \) is added to ensure the coercivity of the bilinear form. 

In Theorem \ref{theorem_Projectionestimate}, we derive an error estimate for the LOD projection error \( \| \mathbf{M}^n - A_{\text{LOD}}(\mathbf{M}^n) \| \) based on the bilinear form in \eqref{bilinear_B}. The proof uses techniques analogous to those in the linear elliptic analysis of Theorem~\ref{theorem_convergence_rate} and is provided in detail in Appendix \ref{subsec:Proof of Theorem ref{theorem_LODestimate}.}.

\begin{theorem}\label{theorem_Projectionestimate}
Assume \eqref{assumption} holds, and regarding $\M^{n+1}$ as the solution to the elliptic system \eqref{time_discrete_form}, the LOD approximation $A_{\LOD}(\M^{n+1}) \in \V_{\LOD}$ adheres to the LOD approximation \eqref{bilinear_B}. Then, there holds
\begin{equation}\label{LODestimate}
    \| \M^n - A_{\LOD} (\M^n)||_{L^2} + H || \M^n - A_{\LOD} (\M^n)||_{H^1} \le C H^3 ||\f^{n+1}||_{H^1}. 
\end{equation}
\end{theorem}

The following theorem demonstrates that the approximation error within the LOD space, \( \| A_{\text{LOD}}(\mathbf{M}^n) - \mathbf{M}_{\text{LOD}}^n \| \), converges with third-order accuracy in the \( L^2 \)-norm and second-order accuracy in the \( H^1 \)-norm with respect to the coarse mesh size \( H \). This contrasts sharply with the standard FEM, which achieves only second-order accuracy in \( L^2 \) and first-order accuracy in \( H^1 \) with respect to the fine mesh size \( h \). A detailed proof is provided in Appendix \ref{subsec:Proof of Theorem ref{theorem_LODestimate}.}.

\begin{theorem}\label{theorem_LODestimate}
Let $T>0$ represent a specified constant, and assume that the LL equation \eqref{LLeq} admits a unique solution that fulfills the assumption \eqref{assumption}. Then, the elliptic system \eqref{time_discrete_form} with homogeneous Neumann boundary condition admits a unique solution $\M^{j+1}$ such that $H \le H_1$ for some $ H_1 >0$. Moreover, the LOD projection $A_{\LOD}(\M^{n+1}) \in \V_{\LOD}$ is consistent with the LOD approximation \eqref{bilinear_B}. 
Denote $\e_{\LOD}^n  :=A_{\LOD} (\M^n) -\M_{\LOD}^n$, then, there holds 
\begin{equation}\label{LODerror}
   || \e_{\LOD}^n||_{L^2} + H ||  \e_{\LOD}^n ||_{H^1} \le C H^3 ,
\end{equation}
where $C$ is a constant independent of the mesh size $H$, but may depend on $\Omega$, $\alpha$, and the regularity constant $C_\mathrm{reg}$ in Assumption \ref{assumption}. Here, the approximation solution $\M_{\LOD}^{n+1} \in \V_{\LOD}$ satisfies 
\begin{equation}\label{436}
	\begin{aligned}
		(D_{\tau} \M_{\LOD}^{n+1},\v) + \alpha(  \nabla \M_{\LOD}^{n+1},\nabla \v)  - (\M_{\LOD}^{n} \times \nabla  \M_{\LOD}^{n+1}, \nabla \v) 
		=\alpha ( (\nabla \M_{\LOD}^{n} \cdot \nabla \M_{\LOD}^{n} )\M_{\LOD}^{n+1}, \v) ,
	\end{aligned}
\end{equation}
for $\v \in \V_{\LOD}.$
\end{theorem}

\subsection{Proof of Theorem \ref{mainresult}}
\label{subsec:Proof of Main Theorem}

We need the following inverse estimate for the LOD space, which was established in \cite[Lemma 10.8]{henning2022superconvergence}.
\begin{lemma}[$H^2$ stability in LOD space]\label{LOD_inverse_estimate}
Assuming that \eqref{assumption} holds, and the bilinear form is defined by \eqref{bilinear_B}.
Then, for any $\M^{n+1} \in \H^2(\Omega)$ and the LOD approximation $A_{\LOD}(\M^{n+1}) \in \V_{\LOD}$ with
\begin{equation}
    B^{n+1}(A_{\LOD}(\M^{n+1}),\v) = B^{n+1}(\M^{n+1},\v)
\end{equation}
for all $\v \in \V_{\LOD}$, and fullfills
\begin{equation}
     ||A_{\LOD}(\M^{n+1})||_{H^2} \le  C||\M^{n+1}||_{H^2} ,
\end{equation}
where $C$ only depends on $B^{n}(\cdot,\cdot)$, $\Omega$ and mesh regularity constants.
    Furthermore, for any $\v_{\LOD} \in \V_{\LOD}$ we have the inverse estimates
    \begin{equation}
        || \v_{\LOD} ||_{H^2} \le C H^{-1} || \v_{\LOD} ||_{H^1}, ~|| \v_{\LOD} ||_{H^1} \le C H^{-1} || \v_{\LOD} ||_{L^2}.
    \end{equation}
\end{lemma}

By combining the results of Theorems \ref{temporal_1}, \ref{theorem_Projectionestimate}, and \ref{theorem_LODestimate}, we prove Theorem \ref{mainresult}, which provides the total (fully discrete) error estimate and is the main result of this paper.

\begin{proof}
    By combining estimates \eqref{time_discre}, \eqref{LODestimate} and \eqref{LODerror} collectively, the total error can be decomposed into two parts:
    \begin{equation}
    \begin{aligned}
    \| \m(t_n) -\M_{\LOD}^n \|_{L^2} &\le \| \m(t_n) -\M^n \|_{L^2} + \| \M^n - \M_{\LOD}^n \|_{L^2} \\
    &\le C \tau + \| \M^n -A_{\LOD}(\M^n) \|_{L^2} +  \| A_{\LOD}(\M^n) - \M_{\LOD}^n \|_{L^2} \\
    &\le C(\tau +  H^3 + H^3 ).
    \end{aligned}
\end{equation}

Using the inverse estimate in Lemma \ref{LOD_inverse_estimate}, we have  
\[
\| A_{\mathrm{LOD}}(\M^n) - \M_{\mathrm{LOD}}^n \|_{H^1} \le C H^{-1} \| A_{\mathrm{LOD}}(\M^n) - \M_{\mathrm{LOD}}^n \|_{L^2}.
\]  
Combined with a similar decomposition, this yields the desired $H^1$ error estimate.
\begin{equation}
    \|\m(t_n) - \M_{\LOD}^n \|_{H^1} 
    \leq C(\tau + H^2 + H^2).
    \label{total_error_H1}
\end{equation}
This completes the proof of Theorem \ref{mainresult}.
\end{proof}

Moreover, the modulus of $\M_{\LOD}^n$ satisfies the following convergence order.
\begin{corollary} 
Under the condition of Theorem \ref{mainresult}, the LOD solution $\M_{\LOD}^n$ satisfies
    \begin{equation*}
        \| 1-|\M_{\LOD}^n|^2 \|_{L^2} \le C (\tau +  H^3).
    \end{equation*}
\begin{proof}
Recall that $\m^n:=\m(\cdot,t_n)$ with $| \m^n |=1$. This estimate follows directly from the identity $$|\m^n|^2-|\M_{\LOD}^n|^2 = (\m^n+\M_{\LOD}^n)(\m^n-\M_{\LOD}^n)$$ together with the uniform bound
    \begin{equation*}
    \| \M_{\LOD}^n \|_{L^{\infty}} \le \| \e_{\LOD}^n \|_{L^{\infty}} + C (\tau + H^2) + \| \m^n \|_{L^{\infty}}\le C.
\end{equation*}     
Finally, applying H\"older's inequality completes the proof.
\end{proof}

\end{corollary}

\section{Numerical Results}
\label{sec: Numerical results}

In this section, we present several examples to demonstrate the performance of the proposed numerical method. We consider the problem in the unit square \(\Omega = [0,1]^2\) with final time \(T\). The LL system is governed by the initial condition \(\mathbf{m}_0\) and homogeneous Neumann boundary conditions, where \(\alpha\) denotes the damping parameter. Since the numerical scheme is first-order accurate in time and third-order accurate in the \(L^2\)-norm in space, the time step should be chosen sufficiently small.

\begin{example}[Constant Coefficient]
In this numerical experiment, the final simulation time is set to $T=0.5$, and the initial condition is specified as $\m_0=(0,0,1)^T$. The time step is defined as $\tau = T/N$, with time step size $\tau = 1e-5$ employed in this simulation. 

To validate the spatial accuracy in two dimensions, we construct an analytical solution of the form
\begin{equation*}
    \m_{e} = (\cos(x^2(1-x)^2y^2(1-y)^2)\sin(t),\sin(x^2(1-x)^2y^2(1-y)^2)\sin(t),\cos(t))^T.
\end{equation*}
The substitution of $\m_{e}$ into the governing equation yields the required forcing term
\begin{equation*}
	\f_e =\partial_t \m_{e} - \alpha \Delta \m_{e}+ \m_e \times \Delta \m_e  - \alpha |\nabla  \m_e |^2\m_e .
\end{equation*}

Tables \ref{tab:L2H1} and \ref{tab:L2H1 1e-2} record the $L^2$ and $H^1$ norm errors computed on a sequence of coarse meshes with $H = 1/2, 1/4, 1/8$, for damping parameters $\alpha = 1.0$ and $\alpha = 1\times10^{-2}$, respectively. For both values of $\alpha$, the numerical results achieve third-order accuracy in the $L^2$ norm and second-order accuracy in the $H^1$ norm, thus validating the convergence rates predicted in Theorem \ref{mainresult}.

\begin{table*}[htbp]
\centering
\caption{The $L^2$ and $H^1$ errors and convergence rate of multiscale basis for a series of coarse meshes satisfying  $H =1/2,1/4,1/8$ with damping parameter $\alpha = 1.0 .$}
\begin{tabular}{|c|c|c|}
\hline
$H$ & $||\m_{e}^N - \M_{\LOD}^N ||_{L^2}$  & $||\m_{e}^N - \M_{\LOD}^N ||_{H^1}$  \\ 
\hline
1/2 & 3.0057e-04  & 2.6551e-03 \\ 
1/4 & 2.9370e-05  & 5.7892e-04  \\ 
1/8 & 3.7779e-06  & 1.4748e-04 \\ 
\hline
Order & 3.1570  & 2.0851 \\
\hline
\end{tabular}
\label{tab:L2H1}
\end{table*}

\begin{table*}[htbp]
\centering
\caption{The $L^2$ and $H^1$ errors and convergence rate of multiscale basis for a series of coarse meshes satisfying  $H =1/2,1/4,1/8$ with damping parameter $\alpha = 1e-2$.}
\begin{tabular}{|c|c|c|}
\hline
$H$ & $||\m_{e}^N - \M_{\LOD}^N ||_{L^2}$ &  $||\m_{e}^N - \M_{\LOD}^N ||_{H^1}$  \\ 
\hline
1/2 & 3.0159e-04  & 2.6554e-03  \\ 
1/4 & 2.9479e-05  & 5.7894e-04  \\ 
1/8 & 4.1658e-06  & 1.4749e-04  \\
\hline
Order & 3.0889  & 2.0851 \\
\hline
\end{tabular}
\label{tab:L2H1 1e-2}
\end{table*}

\end{example}

\begin{example}[Quasi-Periodic Coefficient]
In this example, the material coefficient is defined as
\begin{equation*}
    \kappa(x,y) = (1+0.25\sin(2 \pi x/\varepsilon))\cdot(1+0.25\sin(2 \pi y/\varepsilon)+0.25\sin(2 \sqrt 2 \pi y/\varepsilon)),
\end{equation*}
with $\varepsilon=1/32.$ This coefficient, which is periodic in the $x$-direction but non-periodic in the $y$-direction, was previously investigated using the Heterogeneous Multiscale Method (HMM) in \cite{leitenmaier2022heterogeneous}.  We also use it to validate the effectiveness of our proposed multiscale basis functions.

The final time is set to $T=0.2$ with a damping parameter $\alpha=1e-2$. The initial data is defined as $\m_0=\tilde{\m}_{0}/|\tilde{\m}_{0}|$ where the vector $\tilde{\m}_{0}$ is given by
\begin{equation}\label{initial_data_quasi}
\tilde{\m}_{0} = 
\begin{pmatrix}
0.6 + \exp\bigl(-0.3 \cdot (\cos(2\pi (x - 0.25)) + \cos(2\pi (y - 0.12)))\bigr) \\
0.5 + \exp\bigl(-0.4 \cdot (\cos(2\pi x) + \cos(2\pi (y - 0.4)))\bigr) \\
0.4 + \exp\bigl(-0.2 \cdot (\cos(2\pi (x - 0.81)) + \cos(2\pi (y - 0.73)))\bigr)
\end{pmatrix} .
\end{equation}
Since an analytical solution is unavailable for this problem, we compute a reference solution, denoted $\m_{\text{ref}}^N$, on a fine mesh with $h = 1/2^8$ and $\tau=1e-4$ at the final time step. We use this reference solution as a substitute for the exact solution. Accordingly, the corresponding $L^2$ and $H^1$ errors are presented in Table \ref{tab:L2H1 quasi-perriodic2}. Our results indicate that for $H = 1/2, 1/4, 1/8, 1/16$, the proposed multiscale bases achieve approximately third-order convergence in the $L^2$ norm and second-order convergence in the $H^1$ norm. Figure \ref{fig:quasiperio_coefficient_pair} shows the approximation curves for the three components of $\mathbf{m}^N$ (i.e., $(m_1, m_2, m_3)^\mathrm{T}$), computed using the standard $P_1$ FEM and the LOD method ($H = 1/2^4$, $\tau = 1e-4$), demonstrating the superior performance of the latter.

\begin{table*}[htbp]
\centering
\caption{The $L^2$ and $H^1$ errors and convergence rate of multiscale basis for a series of coarse meshes satisfying  $H =1/2,1/4,1/8,1/16$ with damping parameter $\alpha = 1e-2$.}
\begin{tabular}{|c|c|c|}
\hline
$H$ & $||\m_{\text{ref}}^N - \M_{\LOD}^N ||_{L^2}$  & $||\m_{\text{ref}}^N - \M_{\LOD}^N ||_{H^1}$  \\ 
\hline
1/2 & 1.8732e-01  & 9.5928e-01    \\ 
1/4 & 3.2626e-02 & 2.4836e-01     \\ 
1/8 & 5.1201e-03 & 6.0489e-02     \\ 
1/16 & 3.2545e-04  & 6.8678e-03   \\
\hline
Order & 3.0178  & 2.3416 \\
\hline
\end{tabular}
\label{tab:L2H1 quasi-perriodic2}
\end{table*}

\begin{figure}[htbp]
    \centering
    \begin{minipage}{0.48\textwidth}
        \centering
        \includegraphics[height=5cm,width=6cm]{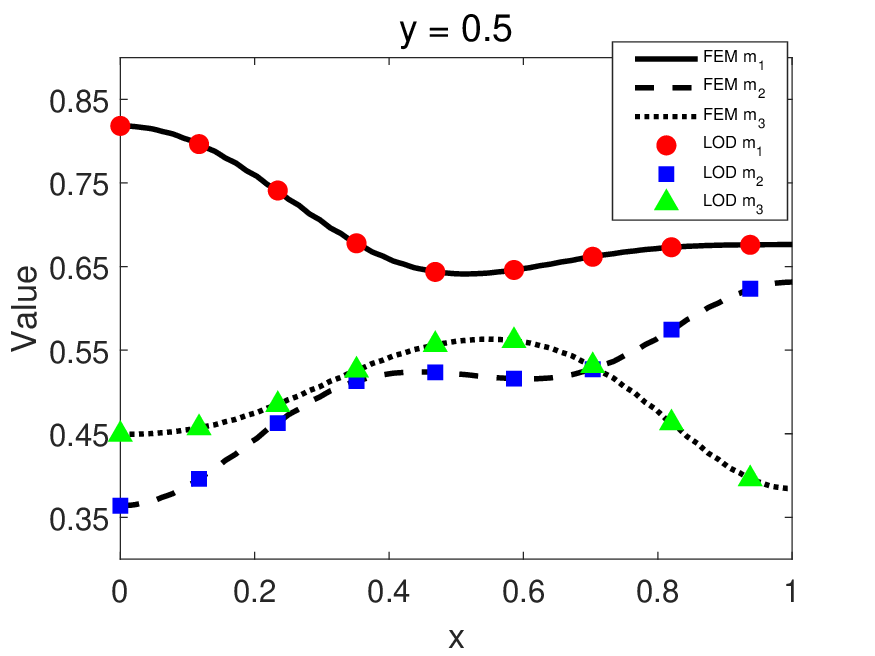}
        \label{fig:quasiperio_coefficient_left}
    \end{minipage}
    \hfill
    \begin{minipage}{0.48\textwidth}
        \centering
        \includegraphics[height=5cm,width=6cm]{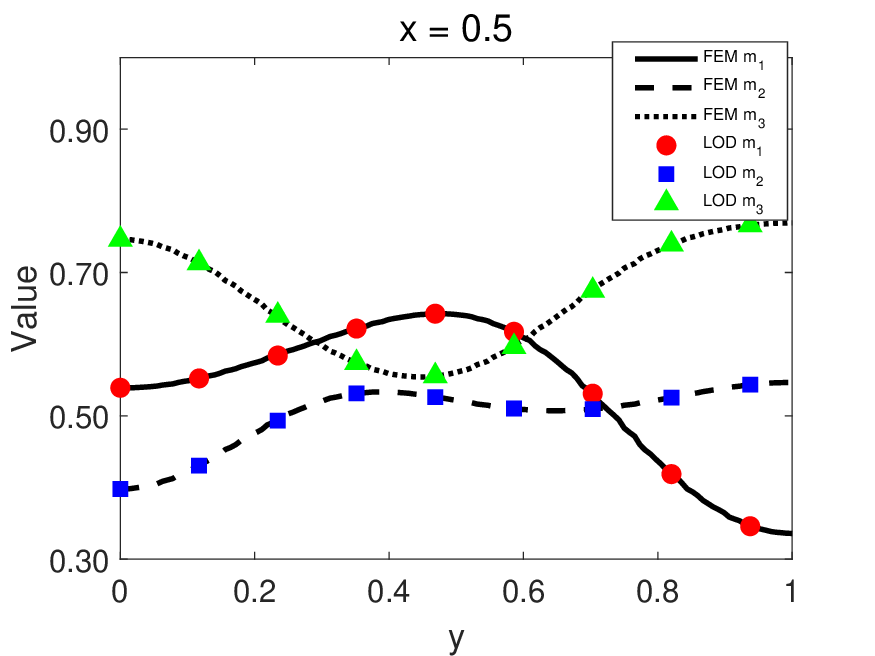}
        \label{fig:quasiperio_coefficient_right}
    \end{minipage}
    \caption{Approximation of $\mathbf{m}$: (Left) $y=0.5$ cross-section, (Right) $x=0.5$ cross-section.} 
    \label{fig:quasiperio_coefficient_pair}
\end{figure}
\end{example}

\begin{example}[Locally-Periodic Coefficient]
The parameters for this simulation (following \cite{leitenmaier2022heterogeneous}) are as follows:
\begin{itemize}
    \item Final time: $T = 5 e{-2}$
    \item Damping parameter: $\alpha = 0.1$
    \item Initial data: $\mathbf{m}_0$ as defined in \eqref{initial_data_quasi}
    \item Material coefficient: $\kappa(x,y)$ takes the form
    \begin{equation}\label{Locally-Periodic}
        \kappa(x,y) = 0.25 \exp\left( -\cos(2\pi (x+y)/\varepsilon) + \sin(2 \pi x/\varepsilon) \cos(2\pi y) \right),
    \end{equation}
    with $\varepsilon=1/64$.
\end{itemize}

Similarly, in the absence of an exact solution, a reference solution computed on a fine mesh ($h = 1/2^9$, $\tau = 5 \times 10^{-4}$) at the final time serves as a benchmark. As shown in Table~\ref{tab:L2H1 Locally-Periodic}, the resulting $L^2$ and $H^1$ errors demonstrate nearly third-order and second-order convergence rates, respectively. Consistent with these findings, Figure \ref{fig:localperio_coefficient_pair} shows the approximation curves for the three components of $\mathbf{m}^N$ (i.e., $(m_1, m_2, m_3)^\mathrm{T}$), computed using the standard $P_1$ FEM and the LOD method ($H = 1/2^4$, $\tau = 5e-4$), demonstrating the superior performance of the latter.

\begin{table*}[htbp]
\centering
\caption{The $L^2$ and $H^1$ errors and convergence rate of multiscale basis for a series of coarse meshes satisfying  $H =1/2,1/4,1/8,1/16$ with damping parameter $\alpha = 1e-2$.}
\begin{tabular}{|c|c|c|}
\hline
$H$ & $||\m_{\text{ref}}^N - \M_{\LOD}^N ||_{L^2}$  & $||\m_{\text{ref}}^N - \M_{\LOD}^N ||_{H^1}$ \\ 
\hline
1/2  & 7.4516e-01    & 5.4159e-01  \\ 
1/4  & 1.3127e-02    & 1.7299e-01  \\ 
1/8  & 2.0868e-03    & 4.8023e-02   \\ 
1/16 & 2.0088e-04    & 9.4035e-03   \\
\hline
Order & 2.8258  & 1.9392 \\
\hline
\end{tabular}
\label{tab:L2H1 Locally-Periodic}
\end{table*}

\begin{figure}[htbp]
    \centering
    \begin{minipage}{0.48\textwidth}
        \centering
        \includegraphics[height=5cm,width=6cm]{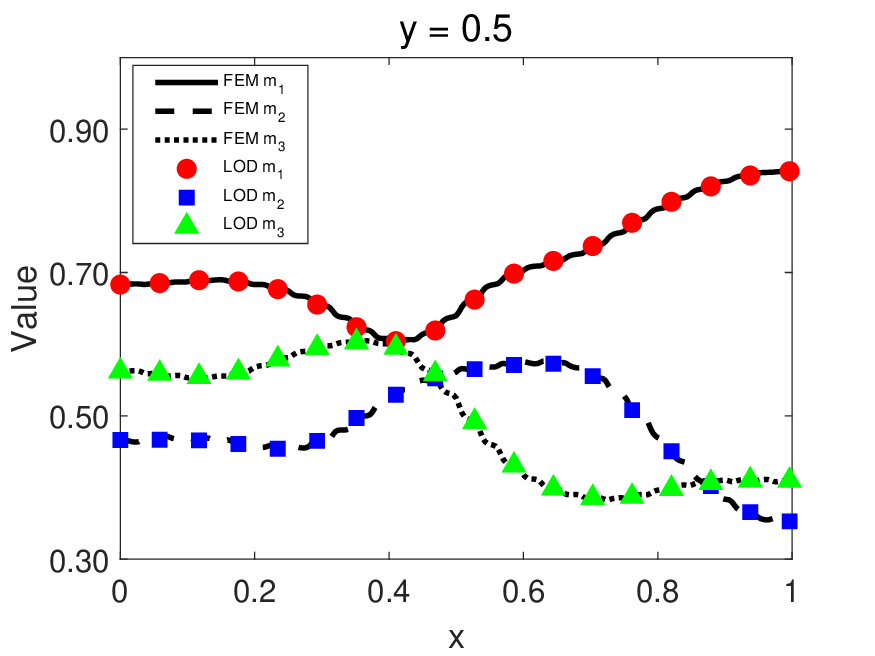}
        \label{fig:localperio_coefficient_left}
    \end{minipage}
    \hfill
    \begin{minipage}{0.48\textwidth}
        \centering
        \includegraphics[height=5cm,width=6cm]{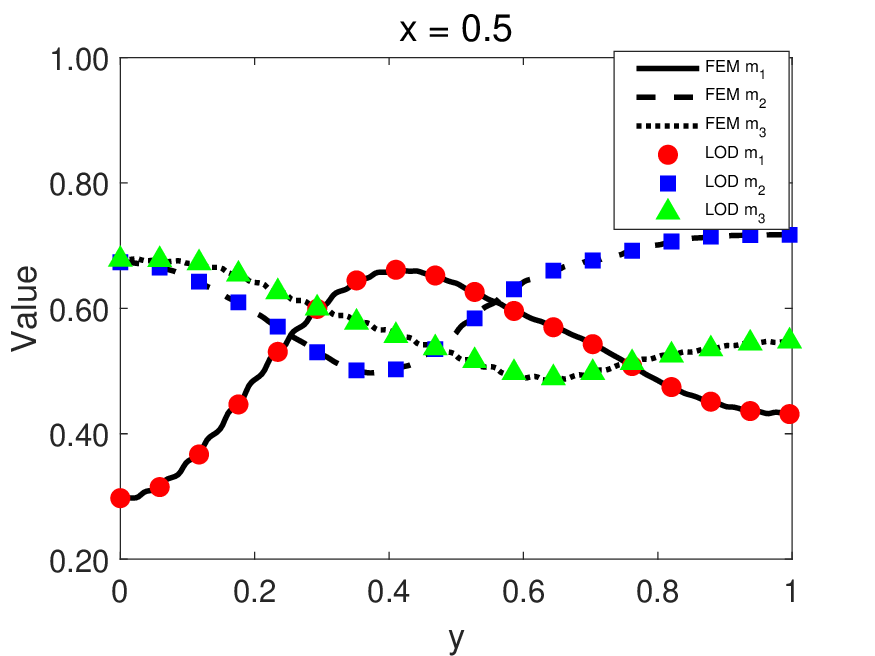}
        \label{fig:localperio_coefficient_right}
    \end{minipage}
    \caption{Approximation of $\mathbf{m}$: (Left) $y=0.5$ cross-section, (Right) $x=0.5$ cross-section.} 
    \label{fig:localperio_coefficient_pair}
\end{figure}
\end{example}

\begin{example}[Rough Coefficient]
In this example, all parameters are kept the same as the quasi-periodic case, except for the material coefficient $\kappa(x,y)$, which is now defined as
\begin{equation}\label{rough_coeff}
    \kappa(x,y)=\text{int}(5+2\sin(2 \pi x)\sin(2 \pi y)),
\end{equation}
where the function “int” rounds a real number $r$ down to the nearest integer less than or equal to $r$.
This definition results in a discontinuous function, as illustrated in Figure~\ref{fig:Rough coefficient}. A similar discontinuous coefficient has been investigated in \cite{henning2017crank}. 

Following the same benchmarking procedure, a reference solution is computed on a fine mesh ($h = 1/2^9$, $\tau = 1\times 10^{-4}$). The resulting $L^2$ and $H^1$ error norms are recorded in Table \ref{tab:L2H1 rough2}. The numerical results, which align with those of the quasi-periodic case, show that the proposed multiscale basis functions achieve approximately third-order convergence in the $L^2$ norm and second-order convergence in the $H^1$ norm for $H = 1/2, 1/4, 1/8, 1/16$. Finally, Figure \ref{Rough coefficient example} displays the profile of the first component $m_1$ and its approximations obtained via the FEM and LOD methods ($H = 1/2^4$, $\tau = 1\times 10^{-4}$) at the final time $T$, demonstrating the excellent approximation capability of the LOD method.
\begin{figure}[htbp]
\centering
\includegraphics[height=5cm, keepaspectratio]{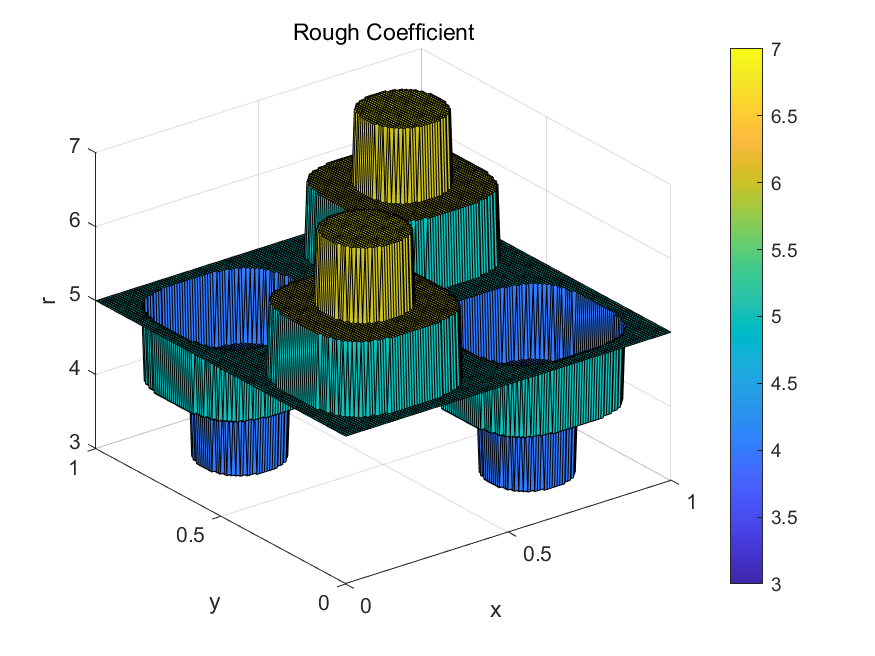}
\caption{Discontinuous rough coefficient $\kappa(x,y)$ on $\Omega=[0,1]^2$.}
\label{fig:Rough coefficient}
\end{figure}

\begin{table*}[htbp]
\centering
\caption{The $L^2$ and $H^1$ errors and convergence rate of multiscale basis for a series of coarse meshes satisfying  $H =1/2,1/4,1/8,1/16$ with damping parameter $\alpha = 1e-2$. }
\begin{tabular}{|c|c|c|}
\hline
$H$ & $||\m_{\text{ref}}^N - \M_{\LOD}^N ||_{L^2}$ & $||\m_{\text{ref}}^N - \M_{\LOD}^N ||_{H^1}$ \\ 
\hline
1/2 & 1.0303e-02 & 7.1185e-02 \\ 
1/4 & 1.5146e-03 & 2.5397e-02 \\ 
1/8 & 1.2056e-04 & 4.8299e-03 \\ 
1/16 & 1.2092e-05 & 1.1537e-03 \\
\hline
Order & 3.2856  & 2.0236 \\
\hline
\end{tabular}
\label{tab:L2H1 rough2}
\end{table*}

\begin{figure}[H]
    \centering
    
    \subfigure[Reference solution: FEM]{
        \includegraphics[width=5cm]{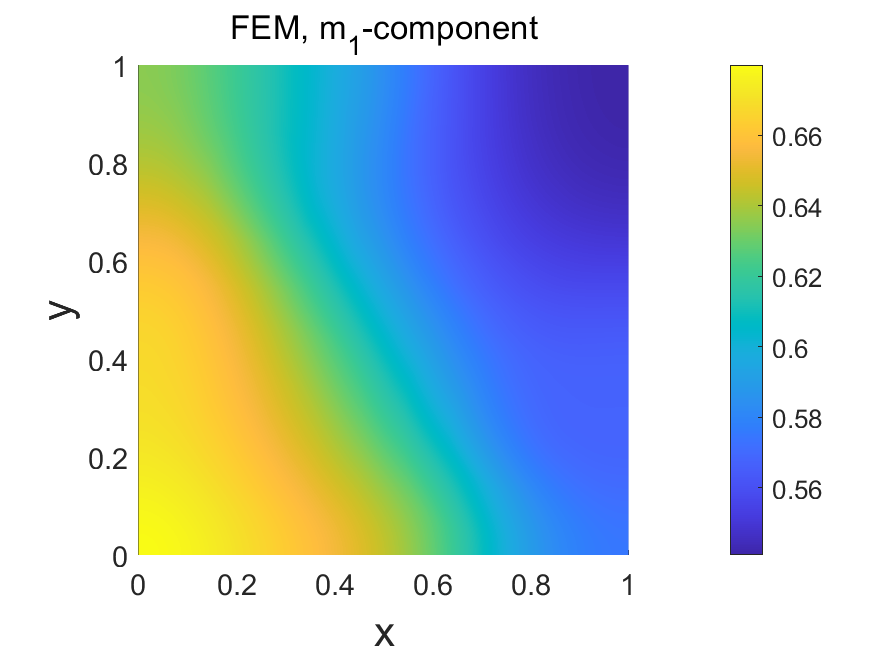}
        \label{FEM_m1_value}
    }
    \quad
    \subfigure[Numerical solution: LOD]{
        \includegraphics[width=5cm]{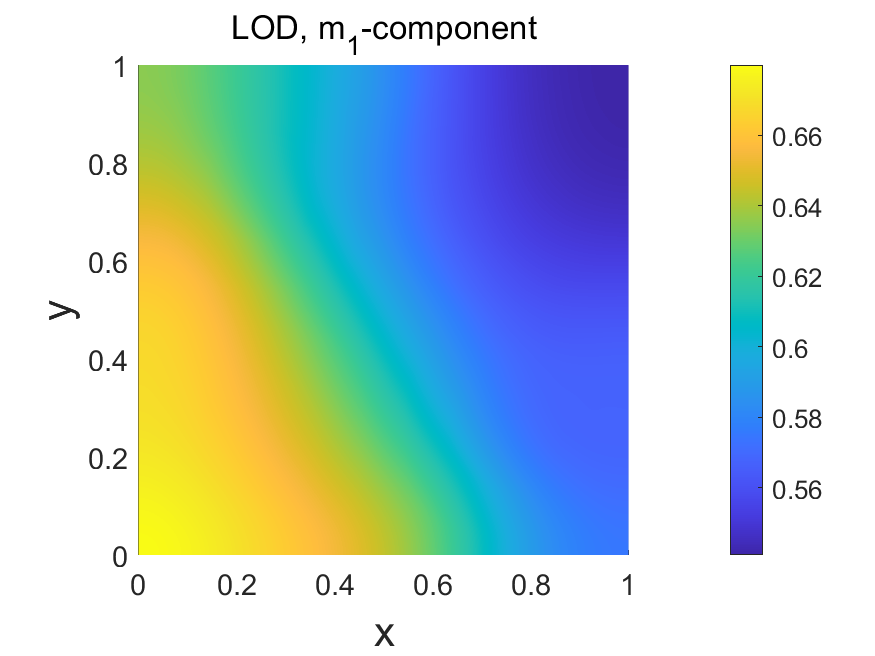}
        \label{LOD_m1_value}
    }
    
    \par\vspace{0.2cm}
    \centering
    {\small(a) Contours}
    \par\vspace{0.3cm}
    
    \subfigure[Approximation curve: $y=0.5$]{
        \includegraphics[width=5cm]{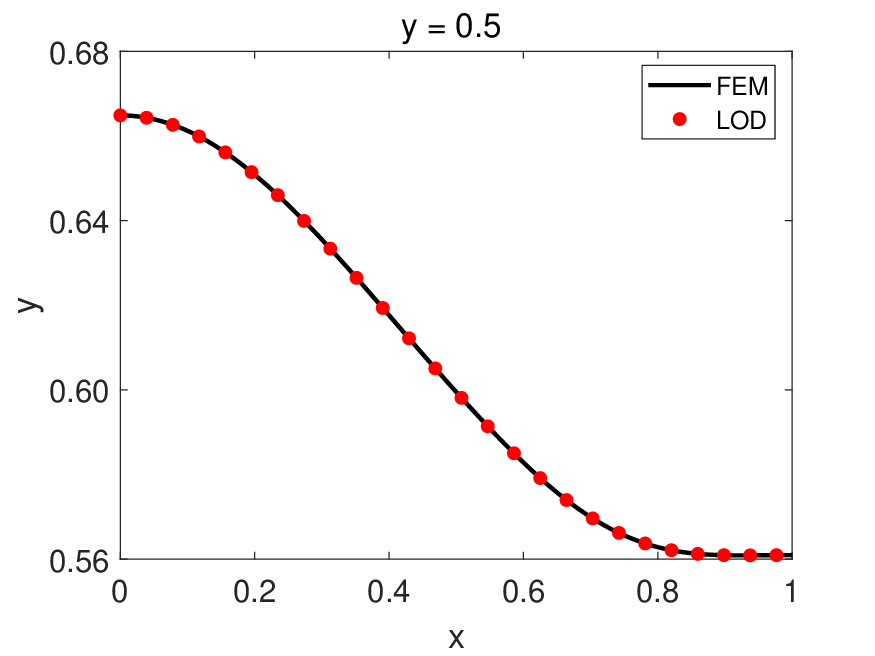}
        \label{Fig2-1expdecay}
    }
    \quad
    \subfigure[Approximation curve: $x=0.5$]{
        \includegraphics[width=5cm]{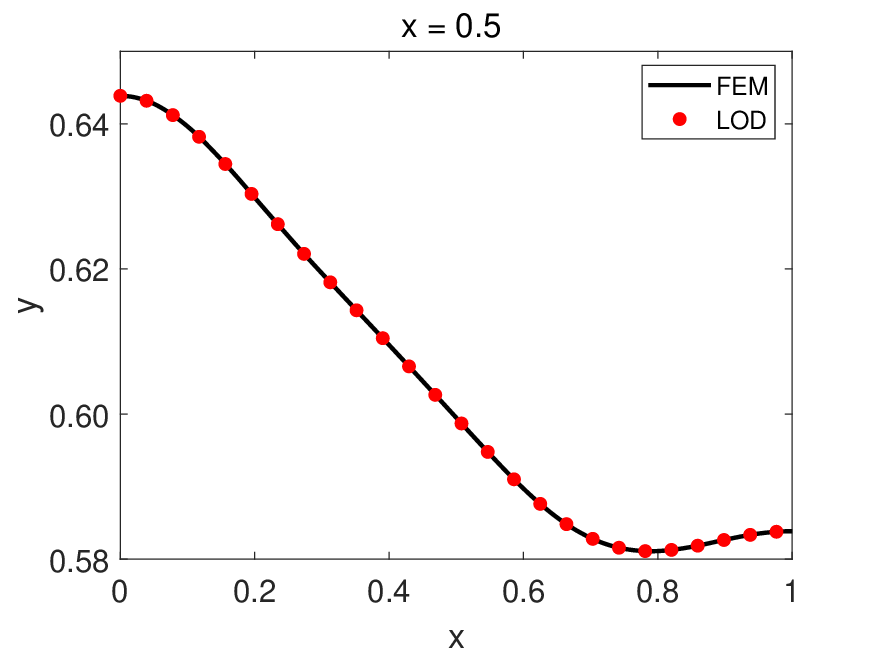}
        \label{Fig2-2expdecay}
    }
    
    \par\vspace{0.2cm}
    \centering
    {\small(b) Cross sections}
    
    \caption{Rough coefficient example, with $T=0.2$ and $\alpha=1e-2$}
    \label{Rough coefficient example}
\end{figure}

\end{example}

\section{Proof of Main Theorems}\label{sec11}

This section is dedicated to proving the main theorems of this work, namely, Theorems \ref{temporal_1}, \ref{theorem_Projectionestimate}, and \ref{theorem_LODestimate}.

\subsection{Proof of Theorem \ref{temporal_1}. }
\label{subsec:Proof of Theorem ref{temporal_1}.}

We use mathematical induction to prove the temporal discretization error estimate.

\begin{proof}
We begin by establishing the existence and uniqueness of the solution \(\mathbf{M}^{n+1}\) to the elliptic system \eqref{time_discrete_elliptic}. This follows directly from the Lax--Milgram theorem \cite{cimrak2005error,gao2014optimal}. Since \(\mathbf{e}^0 = \mathbf{M}^0 - \mathbf{m}^0 = \mathbf{m}_0 - \mathbf{m}_0 = \mathbf{0}\), the estimate \eqref{time_discre} holds trivially for \(n = 0\). We now proceed to prove \eqref{time_discre} for general \(n\) by mathematical induction.

Assume that \eqref{time_discre} holds for all \(k = 1, \dots, n\), i.e.,
\begin{equation}\label{induction_hypothesis}
    \sup_{0 \le k \le n} \Bigl( \|\mathbf{e}^k\|_{H^1}^2 + \tau \sum_{i=0}^{k} \|\mathbf{e}^i\|_{H^2}^2 \Bigr) \le C_0^2 \tau^2 .
\end{equation}
Our goal is to determine a constant \(C_0 > 0\) such that \eqref{time_discre} remains valid for \(k = n+1\).

Testing the error equation \eqref{errorform} with \(\e^{n+1}\) and \(\Delta \mathbf{e}^{n+1}\), respectively, yields the following two expressions 
\begin{equation*}
\begin{aligned}
&(D_{\tau} \e^{n+1},\e^{n+1}) + \alpha \| \nabla \e^{n+1} \|_{L^2}^2
- (\e^{n} \times \nabla \m^{n+1}, \nabla \e^{n+1}) - (\M^{n} \times \nabla \e^{n+1},\nabla \e^{n+1}) \\
&\quad= \alpha(|\nabla \M^{n}|^2 \M^{n+1} - |\nabla \m^{n}|^2 \m^{n+1},\e^{n+1}) - (\R_{\text{tr}}^{n+1}, \e^{n+1}),
\end{aligned}
\end{equation*}

\begin{equation*}
\begin{aligned}
&(D_{\tau} \e^{n+1},\Delta \e^{n+1}) - \alpha\|\Delta \e^{n+1}\|_{L^2}^2 
+ (\nabla \M^{n} \times \nabla \e^{n+1}, \Delta \e^{n+1}) 
+ (\e^{n} \times \Delta \m^{n+1},\nabla \e^{n+1}) \\
&\quad+ (\nabla \e^{n} \times \nabla \m^{n+1}, \Delta \e^{n+1}) = \alpha(|\nabla \M^{n}|^2 \M^{n+1} - |\nabla \m^{n}|^2 \m^{n+1}, \Delta \e^{n+1}) 
- (\R_{\text{tr}}^{n+1}, \Delta \e^{n+1}).
\end{aligned}
\end{equation*}
Here, the truncation error $\R_{\mathrm{tr}}^{n+1}$ is given by \eqref{R_tr}. Special care is required in estimating the following nonlinear terms:
\begin{align}
    &\alpha \bigl( |\nabla \mathbf{M}^{n}|^2 \mathbf{M}^{n+1} - |\nabla \mathbf{m}^{n}|^2 \mathbf{m}^{n+1}, \; \mathbf{e}^{n+1} \bigr), \label{term1} \\
    &\alpha \bigl( |\nabla \mathbf{M}^{n}|^2 \mathbf{M}^{n+1} - |\nabla \mathbf{m}^{n}|^2 \mathbf{m}^{n+1}, \; \Delta \mathbf{e}^{n+1} \bigr). \label{term2}
\end{align}
The remaining terms in the error equation can be handled using standard techniques, such as H\"{o}lder's inequality, as detailed in Theorem 3.1 of \cite{gao2014optimal}. We begin by reformulating the common nonlinear term appearing in the error equation using the following algebraic identity,
\begin{equation}\label{term_common}
    \begin{aligned}
        & |\nabla \M^{n}|^{2} \M^{n+1} - |\nabla \m^{n}|^{2} \m^{n+1} \\
        ={} & |\nabla \M^{n}|^{2} \M^{n+1} - |\nabla \M^{n}|^{2} \m^{n+1} 
              + |\nabla \M^{n}|^{2} \m^{n+1} - |\nabla \m^{n}|^{2} \m^{n+1} \\
        ={} & (\nabla \e^{n} + 2\nabla \m^{n}) \nabla \e^{n} \e^{n+1} 
              + |\nabla \m^{n}|^{2} \e^{n+1} + (\nabla \e^{n} + 2\nabla \m^{n}) \nabla \e^{n} \m^{n+1}.
    \end{aligned}
\end{equation}

Testing the error equation \eqref{errorform} with $\v = \e^{n+1}$ and applying H\"older's inequality, we obtain
\begin{equation}\label{termall}
    \begin{aligned}
        &\bigl(|\nabla \M^{n}|^{2} \M^{n+1} - |\nabla \m^{n}|^{2} \m^{n+1},\; \e^{n+1}\bigr) \\
        \leq{} &\Bigl(\|(\nabla \e^{n} + 2\nabla \m^{n}) \nabla \e^{n} \e^{n+1}\|_{L^{2}} 
                + \||\nabla \m^{n}|^{2} \e^{n+1}\|_{L^{2}}  + \|(\nabla \e^{n} + 2\nabla \m^{n}) \nabla \e^{n} \m^{n+1}\|_{L^{2}}\Bigr) 
                \|\e^{n+1}\|_{L^{2}} \\
        \leq{} &\Bigl(\|\nabla \e^{n}\|_{L^{6}}^{2} \|\e^{n+1}\|_{L^{6}} 
                + \|\nabla \e^{n}\|_{L^{3}} \|\e^{n+1}\|_{L^{6}} 
                + \|\e^{n+1}\|_{L^{2}}  + \|\nabla \e^{n}\|_{L^{3}} \|\nabla \e^{n}\|_{L^{6}} 
                + \|\nabla \e^{n}\|_{L^{2}}\Bigr) \|\e^{n+1}\|_{L^{2}}.
    \end{aligned}
\end{equation}

The terms $\|\e^n\|_{L^3}$ and $\|\nabla \e^n\|_{L^3}$ can be further bounded using the Gagliardo--Nirenberg inequality in Appendix \ref{lem:GN_inequality}. Together with the induction hypothesis \eqref{induction_hypothesis}, we obtain
\begin{align}
    \|\e^n\|_{L^3} &\leq C \|\e^n\|_{L^2}^{1/2} \|\e^n\|_{H^1}^{1/2} \leq C_0 \tau, \label{bound_L3} \\
    \|\nabla \e^n\|_{L^3} &\leq C \|\nabla \e^n\|_{L^2}^{1/2} \|\nabla \e^n\|_{H^1}^{1/2} \leq C_0 \tau^{3/4}. \label{bound_grad_L3}
\end{align}

Collecting the above estimates and applying Young's inequality, \eqref{termall} can be further bounded as
\begin{equation}\label{e_n+1}
    \begin{aligned}
        &\bigl(|\nabla \M^{n}|^{2} \M^{n+1} - |\nabla \m^{n}|^{2} \m^{n+1},\; \e^{n+1}\bigr) \\
        \leq  &\Bigl(\|\nabla \e^{n}\|_{L^{6}}^{2} \|\e^{n+1}\|_{L^{6}} 
              + C_{0} \tau^{\frac{3}{4}} \bigl(\|\e^{n+1}\|_{H^{1}} + \|\nabla \e^{n}\|_{L^{6}}\bigr)  + \|\e^{n+1}\|_{L^{2}} + \|\nabla \e^{n}\|_{L^{2}}\Bigr) \|\e^{n+1}\|_{L^{2}} \\
        \leq & \eps \|\e^{n+1}\|_{H^{1}}^{2} 
              + \eps \|\e^{n+1}\|_{L^{2}}^{2} 
              + \eps \|\e^{n}\|_{H^{2}}^{2}  + \eps \|\e^{n+1}\|_{L^{2}}^{2} 
              + \eps \|\e^{n}\|_{H^{1}}^{2} 
              + \eps^{-1} \|\e^{n+1}\|_{L^{2}}^{2},
    \end{aligned}
\end{equation}
where $\eps > 0$ is chosen such that $C_{0} \tau \leq C_{0} \tau^{3/4} \leq \eps$.

Testing with $\v = \Delta \e^{n+1}$ and following an argument analogous to the previous estimate, we obtain
\begin{equation}\label{termall2}
    \begin{aligned}
        &\bigl( |\nabla \M^{n}|^{2} \M^{n+1} - |\nabla \m^{n}|^{2} \m^{n+1},\; \Delta \e^{n+1} \bigr) \\
        \leq & \bigl\| (\nabla \e^{n} + 2\nabla \m^{n})\nabla \e^{n}(\e^{n+1}+\m^{n+1}) 
              + |\nabla \m^{n}|^{2} \e^{n+1} \bigr\|_{L^{2}}
              \|\Delta \e^{n+1}\|_{L^{2}} \\
        \leq & \eps \|\e^{n}\|_{H^{2}}^{2} 
              + \eps \|\e^{n+1}\|_{H^{2}}^{2} 
              + C \eps^{-1} \|\e^{n}\|_{H^{1}}^{2}.
    \end{aligned}
\end{equation}

Combining the estimates \eqref{e_n+1}, \eqref{termall2} with the bounds established in \cite{gao2014optimal} yields
\begin{equation}\label{combined_estimate}
D_{\tau}\bigl(\|\e^{n+1}\|_{H^1}^2\bigr) + \alpha \|\e^{n+1}\|_{H^2}^2 
        \leq \eps \|\e^{n}\|_{H^2}^2 + \eps \|\e^{n+1}\|_{H^2}^2  + \eps^{-1} \|\e^{n}\|_{H^1}^2 + \eps^{-1} \|\e^{n+1}\|_{H^1}^2 + \eps^{-1} \|\R_{\text{tr}}^{n+1}\|_{L^2}^2.
\end{equation}

Finally, employing the discrete Gronwall's inequality from Lemma \ref{Gronwall}, we deduce that
\begin{equation}\label{final_estimate}
    \|\e^{n+1}\|_{H^1}^2 + \tau \sum_{k=0}^{n} \|\e^k\|_{H^2}^2 
    \leq \exp\!\Bigl(\frac{T}{1-\tau}\Bigr) \tau^2 
    \overset{\tau \leq \tau_1}{\leq} \exp(2TC) \tau^2,
\end{equation}
where $\tau$ is chosen smaller than $\tau_1 = 0.5$ and $\eps$ is taken sufficiently small. Thus, the proof of \eqref{time_discre} is completed by taking $C_0 \geq \exp(TC)$.
\end{proof}

\subsection{Proof of Theorem \ref{theorem_Projectionestimate}. }
\label{subsec:Proof of Theorem ref{theorem_Projectionestimate}.}

This theorem bounds the LOD projection error $\|\M^{n} - A_{\LOD}(\M^{n})\|$ defined by the bilinear form in \eqref{bilinear_B}. The proof extends the convergence arguments established for the linear elliptic case in Theorem~\ref{theorem_convergence_rate}, employing the bilinear form given in \eqref{bilinear_B}.

\begin{proof}
Recall that $D_{\tau} \M^{n+1} = (\M^{n+1} - \M^{n})/\tau$, and Cimr\'{a}k's scheme can be rewritten as the following elliptic system
\begin{equation}\label{another_elliptic_H3}
    \begin{aligned}
        -\alpha \Delta \M^{n+1} + \M^{n} \times \Delta \M^{n+1} 
        &+ \nabla \M^{n} \times \nabla \M^{n+1} + \M^{n+1} \\
        &= -D_{\tau} \M^{n+1} + \M^{n+1} + \alpha |\nabla \M^{n}|^{2} \M^{n+1} \\
        &:= \mathbf{f}^{n+1}.
    \end{aligned}
\end{equation}
Next, with the estimate \eqref{LODestimate} established, we aim to prove that 
\begin{equation}
    \f^{n+1}  \in {H^1} .
\end{equation}
When analyzing the first and second terms of $\mathbf{f}^{n+1}$, they can be constrained by \eqref{MnH2}, \eqref{Dtau}, and possess the property
\begin{equation}
    \|D_{\tau} \M^{n+1}\|_{H^1} \le C .
\end{equation}

For the treatment of the third term, using H\"{o}lder's inequality, we obtain
\begin{equation}
    \begin{aligned}
    \big\||\nabla \mathbf{M}^{n}|^{2} \mathbf{M}^{n+1}\big\|_{H^{1}} 
    &= \big\||\nabla \mathbf{M}^{n}|^{2} \mathbf{M}^{n+1}\big\|_{L^{2}} 
    + \big\|\nabla\big(|\nabla \mathbf{M}^{n}|^{2} \mathbf{M}^{n+1}\big)\big\|_{L^{2}} \\
    &\leq \|\nabla \mathbf{M}^{n}\|_{L^{6}}^{2} \|\mathbf{M}^{n+1}\|_{L^{6}} 
    + \big\| |\nabla \mathbf{M}^{n}|^{2} \nabla \mathbf{M}^{n+1}\big\|_{L^{2}} \quad + \big\|2 \nabla \mathbf{M}^{n} \nabla^{2} \mathbf{M}^{n} \mathbf{M}^{n+1}\big\|_{L^{2}} \\
    &\leq \|\mathbf{M}^{n}\|_{H^{2}}^{2} \|\mathbf{M}^{n+1}\|_{H^{2}} 
    + \|\nabla \mathbf{M}^{n}\|_{L^{6}}^{2} \|\nabla \mathbf{M}^{n+1}\|_{L^{6}}  + \|\nabla \mathbf{M}^{n}\|_{L^{\infty}} \|\nabla^{2} \mathbf{M}^{n}\|_{L^{2}} \|\mathbf{M}^{n+1}\|_{L^{\infty}} \\
    &\leq \|\mathbf{M}^{n}\|_{H^{2}}^{2} \|\mathbf{M}^{n+1}\|_{H^{2}} 
    + \|\mathbf{M}^{n}\|_{H^{2}}^{2} \|\mathbf{M}^{n+1}\|_{H^{2}}  + \|\mathbf{M}^{n}\|_{W^{2,4}} \|\nabla^{2} \mathbf{M}^{n}\|_{L^{2}} \|\mathbf{M}^{n+1}\|_{H^{2}} \\
    &\leq C,
    \end{aligned}
\end{equation}
where we used the result that $|| \nabla^2 \M^n ||_{L^2} \le C || \Delta \M^n ||_{L^2}$, with $\nabla \M^n \cdot \n =\0$ and $\M^n \in H^2$. Furthermore, by Sobolev embeddings, we have
\[
\|\M^{n}\|_{L^{6}} \leq C \|\M^{n}\|_{H^{1}},
\qquad 
\|\M^{n+1}\|_{L^{\infty}} \leq C \|\M^{n+1}\|_{H^{2}}.
\]

Consequently,
\[
\|\M^{n} - A_{\LOD}(\M^{n})\|_{L^{2}} + H\|\M^{n} - A_{\LOD}(\M^{n})\|_{H^1} \leq C H^{3} \|\mathbf{f}^{n+1}\|_{H^{1}}.
\]
\end{proof}

\subsection{Auxiliary Lemmas for Theorem \ref{theorem_LODestimate}}
\label{subsec_lemma:Auxiliary_Lemmas}

To prove Theorem \ref{theorem_LODestimate}, we first establish several auxiliary lemmas. Lemma \ref{lemma52} provides bounds for $\| \M^n \|_{W^{2,4}}$, while Lemmas \ref{lemma_tau1} and \ref{lemma_tau2} estimate the temporal errors $\| \Delta e^{n+1} \|_{L^2}$ and $\| \nabla \Delta e^{n+1} \|_{L^2}$. These higher-order estimates are essential for our framework, as they ensure the \( H^1 \) regularity of the forcing term in the elliptic system \eqref{another_elliptic_H3}, a standard requirement in LOD analysis for nonlinear problems \cite{henning2022superconvergence}. Subsequently, Lemma \ref{lemma_tau3} is derived from these results and plays a central role in the proof of Theorem \ref{theorem_LODestimate}.

\subsubsection{Lemma \ref{lemma52} }
\label{subsec:Proof of Lemma ref{lemma52}}
Under Assumption \ref{assumption}, the exact solution $\m^n$ of the Landau--Lifshitz (LL) equation possesses sufficient regularity. However, the regularity of the temporally discretized solution $\M^n$ remains to be established. The following lemma provides essential regularity estimates for $\M^n$.

\begin{lemma}\label{lemma52}
Under the same assumptions as Theorem \ref{temporal_1}, the following uniform bound holds:
\begin{equation}
    \|\M^{n}\|_{W^{2,4}} \le C.
\end{equation}
\end{lemma}

\begin{proof}
First, from the error estimate \eqref{time_discre} and the regularity assumption \eqref{assumption}, we obtain
\begin{align}
    \sup_{0 \le n \le N} \|\M^{n}\|_{H^{2}} &\le C, \label{MnH2} \\
    \sup_{1 \le n \le N} \|D_{\tau} \M^{n}\|_{H^{1}} &\le C, \label{Dtau} \\
    \tau \sum_{n=1}^{N} \|D_{\tau} \M^{n}\|_{H^{2}}^{2} &\le C. \label{Dtau_sum}
\end{align}

Second, we note that the temporal discretization \eqref{time_discrete_form} of the nonlinear Landau--Lifshitz equation \eqref{LLform1} can be reformulated as a linear elliptic system:
\begin{equation}\label{new_elliptic}
    \begin{aligned}
        - \alpha  \Delta \M^{n+1}  + \M^{n} \times \Delta  \M^{n+1} 
        = - D_{\tau} \M^{n+1} - \nabla \M^{n} \times \nabla \M^{n+1} 
        + \alpha (\nabla  \M^{n} \cdot \nabla \M^{n}) \M^{n+1}.
    \end{aligned}
\end{equation}

Applying the Calder\'on--Zygmund inequality in Appendix \ref{lem:C_Z_inequality} \cite[Section 1.4]{maugeri2000elliptic} to \eqref{new_elliptic} yields
\begin{equation}\label{use_Miranda-Talenti}
\begin{aligned}
\|\M^{n+1}\|_{W^{2,4}} 
&\le C \bigl\|{-}D_{\tau}\M^{n+1} - \nabla\M^{n} \times \nabla\M^{n+1} 
        + \alpha (\nabla\M^{n} \cdot \nabla\M^{n})\M^{n+1}\bigr\|_{L^{4}} \\
&\le C \|D_{\tau}\M^{n+1}\|_{H^{1}} 
        + C \|\nabla\M^{n}\|_{L^{6}} \bigl(\|\nabla\M^{n+1}\|_{L^{12}} 
        + \|\nabla\M^{n}\|_{L^{12}} \|\M^{n+1}\|_{L^{\infty}}\bigr) \\
&\le C + C \|\nabla\M^{n+1}\|_{L^{12}} + C \|\nabla\M^{n}\|_{L^{12}} \\
&\le C + C \bigl( \|\M^{n+1}\|_{W^{2,4}}^{5/7} \|\M^{n+1}\|_{H^{2}}^{2/7} 
                + \|\M^{n}\|_{W^{2,4}}^{5/7} \|\M^{n}\|_{H^{2}}^{2/7} \bigr) \\
&\le \tfrac{1}{2} \|\M^{n+1}\|_{W^{2,4}} + \tfrac{1}{4} \|\M^{n}\|_{W^{2,4}} + C.
\end{aligned}
\end{equation}

The estimates are derived by applying the following steps in sequence:  
(1) the triangle inequality and Hölder’s inequality,  
(2) the uniform bounds on \(\|D_{\tau}\M^{n+1}\|_{H^{1}}\) and \(\|\nabla\M^{n}\|_{L^{6}}\),  
(3) the Gagliardo–Nirenberg interpolation inequality, and  
(4) Young’s inequality with suitably chosen exponents.

Rearranging \eqref{use_Miranda-Talenti} yields  
\[
2 \|\M^{n+1}\|_{W^{2,4}} \le \|\M^{n}\|_{W^{2,4}} + C .
\]

Iterating this inequality gives  
\[
\begin{aligned}
2 \|\M^{n+1}\|_{W^{2,4}} 
&\le \frac{1}{2^{n}} \|\M^{0}\|_{W^{2,4}} + C \sum_{k=0}^{n} \frac{1}{2^{k}} \\
&\le \|\M^{0}\|_{W^{2,4}} + 2C ,
\end{aligned}
\]
where we have used \(\M^{0} = \m_{0}\) together with the initial bound \(\|\m_{0}\|_{W^{2,4}} \le C\) from \eqref{assumption}.  
Consequently,  
\[
\sup_{0 \le n \le N} \|\M^{n}\|_{W^{2,4}} \le C ,
\]
which completes the proof.
\end{proof}

\subsubsection{Lemma \ref{lemma_tau1}}
\label{subsec:Proof of Lemma ref{lemma_tau1}.}

The following lemmas are essential for the nonlinear analysis of the LOD method in the vectorial setting, in contrast to the scalar case treated in \cite[Lemma 10.6]{henning2022superconvergence}, and are not needed in the standard finite element error analysis \cite{gao2014optimal}.

\begin{lemma}\label{lemma_tau1}
Under the same assumptions as in Theorem \ref{temporal_1}, there exists a $\tau$-independent constant $C_0 > 0$ such that
\begin{equation} \label{litte_aim1}
\sup_{0 \le n \le N} \left( \|\Delta e^{n+1}\|_{L^2}^2 + \tau \sum_{k=0}^{n} \|\nabla \Delta e^{k+1}\|_{L^2}^2 \right) \le C_0^2 \tau^2,
\end{equation}
where $e^{n+1}$ denotes the temporal error at time level $t_{n+1}$.
\end{lemma}

\begin{proof}
We continue the proof of \eqref{litte_aim1} by mathematical induction. Differentiating the error equation \eqref{errorform} with respect to $\mathbf{x}$ and testing the resulting equation with $-\nabla \Delta e^{n+1}$, we obtain
\begin{equation}
    \begin{aligned}
    & ( D_{\tau} \Delta \e^{n+1}, \Delta \e^{n+1} ) + \alpha(\nabla \Delta \e^{n+1},\nabla \Delta \e^{n+1}) \\
    =& \underbrace{(\nabla  (  \nabla \M^n \times \nabla \e^{n+1} + \M^n \times \Delta \e^{n+1} ),\nabla \Delta \e^{n+1})  }_{:=T_1}  + \underbrace{(\nabla  (\nabla \e^n \times \nabla \m^{n+1} + \e^n \times  \Delta \m^{n+1} ),\nabla \Delta \e^{n+1}) }_{:=T_2} \\
    & + \underbrace{\alpha (\nabla(|\nabla \M^{n}|^2 \M^{n+1} - |\nabla \m^{n}|^2 \m^{n+1}), \nabla \Delta \e^{n+1})}_{:=T_3} + \alpha \underbrace{(\nabla \R_{\text{tr}}^{n+1}, \nabla \Delta \e^{n+1})}_{:=T_4} .
    \end{aligned}
\end{equation}
Moreover, from the boundary condition $\nabla \e^{n+1} \cdot \mathbf{n} = 0$, we deduce that $D_{\tau} \nabla \e^{n+1} \cdot \mathbf{n} = 0$ on $\partial\Omega$. Next, we estimate the four terms $T_1$--$T_4$ sequentially.

\noindent \emph{Estimate of $T_1$}: Applying Leibniz's rule yields
\begin{equation*}
    \begin{aligned}
    & \nabla  (  \nabla \M^n \times \nabla \e^{n+1} + \M^n \times \Delta \e^{n+1})\\
    =& \nabla^2 \M^n \times \nabla \e^{n+1} + \nabla \M^n \times \nabla^2 \e^{n+1} + \nabla \M^n \times \Delta \e^{n+1} +  \M^n \times \nabla \Delta \e^{n+1} .
    \end{aligned}
\end{equation*}

Then, a direct application of H\"{o}lder's inequality and the Sobolev embedding $H^{1}(\Omega) \hookrightarrow L^{4}(\Omega)$ gives
\begin{equation}
    \begin{aligned}
        T_{1} &= \bigl(\nabla^{2}\M^{n} \times \nabla\e^{n+1} 
               + \nabla\M^{n} \times \nabla^{2}\e^{n+1}  + \nabla\M^{n} \times \Delta\e^{n+1} 
               + \M^{n} \times \nabla\Delta\e^{n+1},\, \nabla\Delta\e^{n+1}\bigr) \\
        &\leq \bigl(\|\nabla^{2}\M^{n}\|_{L^{4}} \|\nabla\e^{n+1}\|_{L^{4}} 
       + \|\nabla\M^{n}\|_{L^{\infty}} (\|\nabla^{2}\e^{n+1}\|_{L^{2}} + \|\Delta\e^{n+1}\|_{L^{2}})\bigr)
       \|\nabla\Delta\e^{n+1}\|_{L^{2}} \\
        &\leq C \Bigl(\|\M^{n}\|_{W^{2,4}} \|\nabla\e^{n+1}\|_{H^{1}} 
               + \|\nabla\M^{n}\|_{L^{\infty}} \|\Delta\e^{n+1}\|_{L^{2}}\Bigr) 
               \|\nabla\Delta\e^{n+1}\|_{L^{2}},
    \end{aligned}
\end{equation}
where $\| \M^{n}\|_{W^{2,4}}$ and $\|\nabla \M^{n}\|_{L^{\infty}}$ are bounded by Lemma \ref{lemma52}. By Theorem \ref{temporal_1}, $\|\nabla \e^{n+1}\|_{L^{2}}^{2} \leq C \tau^{2}$, and elliptic regularity yields $\|\nabla^{2} \e^{n+1}\|_{L^{2}} \leq C \|\Delta \e^{n+1}\|_{L^{2}}$.

Furthermore, applying Young's inequality and $H^1$ norm of $||\nabla \e^{n+1}||_{H^1}$, we bound $T_1$ by 
\begin{equation}
    \begin{aligned}
    T_1\le & \eps^{-1}  ||\nabla \e^{n+1}||_{H^1} ^2 + \eps^{-1}  ||\Delta \e^{n+1} ||_{L^2}^2 + \eps ||\nabla \Delta \e^{n+1} ||_{L^2}^2 \\
    \le &  \eps^{-1} \Bigl( || \nabla \e^{n+1}||_{L^2}^2 + || \nabla^2 \e^{n+1}||_{L^2} ^2 \Bigr) + \eps^{-1}  ||\Delta \e^{n+1} ||_{L^2}^2  + \eps ||\nabla \Delta \e^{n+1} ||_{L^2}^2 \\
    \le &  \eps^{-1}  \tau^2  + \eps^{-1}  ||\Delta \e^{n+1} ||_{L^2}^2  + \eps ||\nabla \Delta \e^{n+1} ||_{L^2}^2 ,
    \end{aligned}
\end{equation}

\noindent \emph{Estimate of $T_2$}: By Theorem \ref{temporal_1}, $\| \e^{n}\|_{L^{2}} \leq C \tau$. Applying Leibniz's rule and H\"older's inequality,
\begin{equation}
    \begin{aligned}
        T_{2} &= \bigl(\nabla(\nabla \e^{n} \times \nabla \m^{n+1} 
               + \e^{n} \times \Delta \m^{n+1}),\, \nabla \Delta \e^{n+1}\bigr) \\
        &= \bigl(\nabla^{2} \e^{n} \times \nabla \m^{n+1} 
               + \nabla \e^{n} \times \nabla^{2} \m^{n+1}  + \nabla \e^{n} \times \Delta \m^{n+1} 
               + \e^{n} \times \nabla \Delta \m^{n+1},\, \nabla \Delta \e^{n+1}\bigr) \\
        &\leq \Bigl(\|\nabla^{2} \e^{n}\|_{L^{2}} \|\nabla \m^{n+1}\|_{L^{\infty}} 
               + \|\nabla \e^{n}\|_{L^{4}} \|\nabla^{2} \m^{n+1}\|_{L^{4}} 
     + \|\nabla \e^{n}\|_{L^{4}} \|\Delta \m^{n+1}\|_{L^{4}}\Bigr) 
               \|\nabla \Delta \e^{n+1}\|_{L^{2}} \\
        &\quad + \| \e^{n}\|_{L^{\infty}} \|\nabla \Delta \m^{n+1}\|_{L^{2}} 
               \|\nabla \Delta \e^{n+1}\|_{L^{2}}.
    \end{aligned}
\end{equation}

Using the Gagliardo-Nirenberg inequality, $\|\nabla \m^{n+1}\|_{L^{\infty}} \leq \| \m^{n+1}\|_{W^{2,4}}$, and similarly $\|\nabla^{2} \m^{n+1}\|_{L^{4}} \leq \| \m^{n+1}\|_{W^{2,4}}$, $\|\Delta \m^{n+1}\|_{L^{4}} \leq \| \m^{n+1}\|_{W^{2,4}}$. Thus,
\begin{equation}
    \begin{aligned}
        T_{2} &\leq \bigl(\|\nabla^{2} \e^{n}\|_{L^{2}} 
               + 2\|\nabla \e^{n}\|_{L^{4}}\bigr) \|\nabla \Delta \e^{n+1}\|_{L^{2}}  + \| \e^{n}\|_{L^{\infty}} \|\nabla \Delta \m^{n+1}\|_{L^{2}} 
               \|\nabla \Delta \e^{n+1}\|_{L^{2}} \\
              &\leq C \bigl(\|\Delta \e^{n}\|_{L^{2}} + \| \e^{n}\|_{H^{2}}\bigr) 
               \|\nabla \Delta \e^{n+1}\|_{L^{2}} \\
              &\leq \epsilon^{-1} \|\Delta \e^{n}\|_{L^{2}}^{2} 
               + \epsilon^{-1} \tau^{2} 
               + \epsilon \|\nabla \Delta \e^{n+1}\|_{L^{2}}^{2},
    \end{aligned}
\end{equation}
where the second inequality uses the Sobolev embedding $\|\nabla \e^{n}\|_{L^{4}} \leq C \|\nabla \e^{n}\|_{H^{1}} \leq C \| \e^{n}\|_{H^{2}}$ and the estimate $\| \e^{n}\|_{H^{2}} \leq C(\| \e^{n}\|_{L^{2}} + \|\Delta \e^{n}\|_{L^{2}})$. Moreover, the boundedness of $\nabla \Delta \m^{n+1} := \nabla \Delta \m(t_{n+1})$ in $L^{2}$ follows from the identity
\[
\nabla \Delta \m(t_{n+1}) - \nabla \Delta \m(t_{n}) = \int_{t_{n}}^{t_{n+1}} \partial_{t} \nabla \Delta \m(t) \, \mathrm{d}t,
\]
which yields
\begin{equation}\label{mnaDel_m}
\begin{aligned}
\|\nabla \Delta \m(t_{n+1})\|_{L^{2}} 
&\le \sum_{k=0}^{n} \|\nabla \Delta \m^{k+1} - \nabla \Delta \m^{k}\|_{L^{2}} 
   + \|\nabla \Delta \m_0\|_{L^{2}} \\
&\le \sum_{k=0}^{n} \Bigl\| \int_{t_k}^{t_{k+1}} \partial_t \nabla \Delta \m(t) \, \d t \Bigr\|_{L^{2}} 
   + \|\nabla \Delta \m_0\|_{L^{2}} \\
&\le \tau N \sup_{t \in [0,T]} \|\partial_t \nabla \Delta \m(t)\|_{L^{2}} 
   + \| \m_0\|_{H^{3}} \\
&\le C T \|\partial_t \m\|_{L^{\infty}(0,T; H^{3})} + \| \m_0\|_{H^{3}} \le C,
\end{aligned}
\end{equation}
where $T = N\tau$ is the final time. The uniform bound is ensured by Assumption \ref{assumption}, which guarantees that 
$\|\partial_{t} \m\|_{L^{\infty}(0,T; H^{3}(\Omega))}$ and $\| \m_0\|_{H^{3}(\Omega)}$ are finite. Applying Young's inequality then gives
\begin{equation}
    \begin{aligned}
     T_2 \le \eps^{-1}  \|\Delta \e^n\|_{L^2}^2 + \eps^{-1} \tau^2 + \eps \|\nabla \Delta \e^{n+1}\|_{L^2}^2 .
    \end{aligned}
\end{equation}

\noindent \emph{Estimate of $T_3$}: Applying Leibniz's rule,
\begin{equation}\label{eq:grad_expansion}
    \begin{aligned}
    \nabla \bigl( |\nabla \M^{n}|^{2} \M^{n+1} - |\nabla \m^{n}|^{2} \m^{n+1} \bigr) = & 
 2\nabla^{2} \M^{n} \nabla \M^{n} \, \e^{n+1} 
       + |\nabla \M^{n}|^{2} \nabla \e^{n+1} 
    + 2\nabla^{2} \e^{n} \nabla \e^{n} \, \m^{n+1} 
       + |\nabla \e^{n}|^{2} \nabla \m^{n+1} \\
    &+ 2\nabla^{2} \e^{n} \nabla \m^{n} \m^{n+1} 
       + 2\nabla \e^{n} \nabla^{2} \m^{n} \, \m^{n+1} 
    + 2\nabla \e^{n} \nabla \m^{n} \nabla \m^{n+1}.
    \end{aligned}
\end{equation}
Substituting \eqref{eq:grad_expansion} into $T_3$ and applying H\"older's inequality term by term,
\begin{equation}\label{eq:T3_holder}
\begin{aligned}
T_3 &= \alpha \bigl( \nabla (|\nabla \M^{n}|^{2} \M^{n+1} - |\nabla \m^{n}|^{2} \m^{n+1}), \nabla \Delta \e^{n+1} \bigr) \\
&\le \alpha \Bigl[ \bigl( \|\nabla^{2} \M^{n}\|_{L^{4}} \|\nabla \M^{n}\|_{L^{\infty}} \| \e^{n+1}\|_{L^{4}} 
       + \|\nabla \M^{n}\|_{L^{\infty}}^{2} \|\nabla \e^{n+1}\|_{L^{2}} \\
&\quad + \|\nabla^{2} \e^{n}\|_{L^{2}} \|\nabla \e^{n}\|_{L^{\infty}} \| \m^{n+1}\|_{L^{\infty}} 
       + \|\nabla \e^{n}\|_{L^{2}} \|\nabla \e^{n}\|_{L^{\infty}} \|\nabla \m^{n+1}\|_{L^{\infty}} \bigr) \\
&\quad + \bigl( \|\nabla^{2} \e^{n}\|_{L^{2}} \|\nabla \m^{n}\|_{L^{\infty}} \| \m^{n+1}\|_{L^{\infty}} 
       + \|\nabla \e^{n}\|_{L^{4}} \|\nabla^{2} \m^{n}\|_{L^{4}} \| \m^{n+1}\|_{L^{\infty}} \\
&\quad + \|\nabla \e^{n}\|_{L^{2}} \|\nabla \m^{n}\|_{L^{\infty}} \|\nabla \m^{n+1}\|_{L^{\infty}} \bigr) \Bigr] 
       \|\nabla \Delta \e^{n+1}\|_{L^{2}}.
\end{aligned}
\end{equation}

Since $\M^n, \m^n \in W^{2,4}$, we have $\|\nabla^2 \M^n\|_{L^4} \le C$, $\|\nabla \M^n\|_{L^{\infty}} \le C$, and similarly for $\m^n$. Using $\|\nabla^{2} \e^{n}\|_{L^{2}} \le C \|\Delta \e^{n}\|_{L^{2}}$ and the Sobolev embedding $H^{1} \hookrightarrow L^{4}$,
\begin{equation}\label{eq:T3_estimate}
        T_3 \le \bigl( \| \e^{n+1}\|_{H^{1}} 
              + \|\Delta \e^{n}\|_{L^{2}} \|\nabla \e^{n}\|_{L^{\infty}}
              + \|\nabla \e^{n}\|_{H^{1}} 
              + \|\Delta \e^{n}\|_{L^{2}} + \|\nabla \e^{n}\|_{L^{2}} \|\nabla \e^{n}\|_{L^{\infty}}
              + \|\nabla \e^{n}\|_{L^{2}} \bigr) 
              \|\nabla \Delta \e^{n+1}\|_{L^{2}}. 
\end{equation}

By the induction hypothesis, $\|\Delta \e^{n}\|_{L^{2}} \le C_{0} \tau \le \epsilon$ with $C_{0} \tau \le \epsilon$. Theorem \ref{temporal_1} gives
\[
\| \e^{n}\|_{H^{1}} \le C \tau \quad \text{and} \quad \| \e^{n+1}\|_{H^{1}} \le C \tau.
\]
The Sobolev embedding $H^{2}(\Omega) \hookrightarrow L^{\infty}(\Omega)$ yields
\[
\|\nabla \e^{n}\|_{L^{\infty}} \le C \|\nabla \e^{n}\|_{H^{2}} 
\le C \bigl( \|\nabla \e^{n}\|_{L^{2}} + \|\nabla \Delta \e^{n}\|_{L^{2}} \bigr),
\]
which is used in \eqref{eq:T3_estimate}. This simplifies to
\begin{equation}
    \begin{aligned}
        T_3 &\le \bigl( \tau 
              + \|\Delta \e^{n}\|_{L^{2}} 
              + \|\nabla \e^{n}\|_{L^{2}} \|\nabla \e^{n}\|_{L^{\infty}}
              + \|\nabla \e^{n}\|_{H^{1}} \bigr) 
              \|\nabla \Delta \e^{n+1}\|_{L^{2}} \\
        &\le \bigl( \tau 
              + \|\Delta \e^{n}\|_{L^{2}} 
              + \epsilon \|\nabla \e^{n}\|_{H^{2}} \bigr) 
              \|\nabla \Delta \e^{n+1}\|_{L^{2}} \\
        &\le \bigl( \tau 
              + \|\Delta \e^{n}\|_{L^{2}} + \epsilon \|\nabla \Delta \e^{n}\|_{L^{2}} \bigr) 
              \|\nabla \Delta \e^{n+1}\|_{L^{2}} .
    \end{aligned}
\end{equation}
Finally, applying Young's inequality,
\begin{equation}
    T_3 \le \epsilon^{-1} \tau^{2} 
        + \epsilon^{-1} \|\Delta \e^{n}\|_{L^{2}}^{2} 
        + \epsilon \|\nabla \Delta \e^{n+1}\|_{L^{2}}^{2}.
\end{equation}

\noindent \emph{Estimate of $T_4$}:
Applying Young's inequality and the truncation error bound from Remark \ref{truncation},
\begin{equation}\label{eq:T4_estimate}
    T_4 \le \epsilon^{-1} \|\nabla \R_{\text{tr}}^{n+1}\|_{L^2}^2 + \epsilon \|\nabla \Delta \e^{n+1}\|_{L^{2}}^{2}.
\end{equation}

Combining all estimates and choosing $\eps$ sufficiently small,
\begin{equation*}
    \|\Delta \e^{n+1}\|_{L^2}^2 + \tau \sum_{k=0}^{n} \|\nabla \Delta \e^{k+1}\|_{L^2}^2 
    \le \eps^{-1} \tau^2 + \eps^{-1} \sum_{k=0}^{n} \|\Delta \e^{k}\|_{L^2}^2.
\end{equation*}
Gronwall's inequality (Lemma \ref{Gronwall}) then yields \eqref{litte_aim1}, and \eqref{time_discre} follows by taking $C_0 \geq \exp(TC)$.
\end{proof}

\subsubsection{Lemma \ref{lemma_tau2}}
\label{subsec:Proof ofref{lemma_tau2}.}

We now prove a higher-order version of Lemma \ref{lemma_tau1}. Applying $\nabla \Delta$ to \eqref{errorform} yields a polyharmonic boundary value problem. Appropriate higher-order boundary conditions, such as $\partial_{\n} (-\nabla \Delta) e^{n+1} = 0$, must be imposed on $\partial\Omega$. The following lemma addresses this setting.

\begin{lemma}\label{lemma_tau2}
Under the same assumptions as in Theorem \ref{temporal_1}, there exists a $\tau$-independent $C_0 > 0$ such that
    \begin{equation}\label{litte_aim2}
    \sup_{0 \le n \le N}( ||\nabla \Delta \e^{n+1}||_{L^2}^2 + \tau \sum \limits _{k=0} ^{n} ||\Delta^2 \e^{k+1}||_{L^2}^2 ) \le C_0^2 \tau^2 . 
\end{equation}
\end{lemma}

\begin{proof}
Similarly, we will prove \eqref{litte_aim2} by mathematical induction. 
First, performing $\Delta$ on \eqref{errorform}, we obtain
\begin{equation}\label{interier}
	\begin{aligned}
		&D_{\tau} \Delta \e^{n+1}  - \alpha  \Delta^2 \e^{n+1}  + \Delta (\M^{n} \times \Delta  \e^{n+1} +   \nabla \M^{n} \times \nabla \e^{n+1} + \e^{n} \times \Delta  \m^{n+1} + \nabla \e^{n} \times \nabla \m^{n+1})\\
		=&\alpha \Delta ( |\nabla  \M^{n} |^2 \M^{n+1} -|\nabla  \m^{n} |^2 \m^{n+1} ) - \Delta \R_{\text{tr}}^{n+1},
	\end{aligned}
\end{equation}
for every $\x \in \Omega.$ Above equation \eqref{interier} can be rewritten as a function $h$ with variable $\e^{n+1}$ satisfies
\begin{equation}
h(\e^{n+1}) = \0 ,
\end{equation}
for every $\x \in \Omega.$ 

Then, applying the operator $\nabla \Delta$ to \eqref{errorform} with respect to $\mathbf{x}$ and testing the resulting equation with $\nabla \Delta \mathbf{e}^{n+1}$, we obtain via integration by parts that
\begin{equation}
    \begin{aligned}
        &(D_{\tau} \Delta \e^{n+1},\, \Delta^{2} \e^{n+1}) 
         + \alpha \|\Delta^{2} \e^{n+1}\|^{2} \\
        ={} &\underbrace{\bigl\langle \Delta(\M^{n} \times \Delta \e^{n+1} + \nabla \M^{n} \times \nabla \e^{n+1}) \bigr\rangle}_{:=J_1}
         + \underbrace{\bigl\langle \Delta(\e^{n} \times \Delta \m^{n+1} + \nabla \e^{n} \times \nabla \m^{n+1}) \bigr\rangle}_{:=J_2} \\
        &\quad + \alpha \underbrace{\bigl\langle \Delta(|\nabla \M^{n}|^{2} \M^{n+1} - |\nabla \m^{n}|^{2} \m^{n+1}) \bigr\rangle}_{:=J_3}
         + \alpha \underbrace{\bigl\langle \Delta \R_{\mathrm{tr}}^{n+1} \bigr\rangle}_{:=J_4}  + \int_{\partial \Omega} (\nabla \Delta \e^{n+1} \cdot \n) h(\e^{n+1}) \, \d\S,
    \end{aligned}
\end{equation}
where we denote $\langle \cdot \rangle := ( \cdot ,\, \Delta^{2} \e^{n+1})$ for brevity.
By \cite[Theorem 1.5.1.2]{grisvard2011elliptic}, the identity \eqref{interier} extends to the boundary, i.e., it holds for all $\x \in \partial \Omega$. Consequently, the boundary integral vanishes.

We now establish the temporal estimates for the three terms individually. 
First, noting that $\Delta = \nabla \cdot \nabla$ and employing the Leibniz's rule, we obtain
\begin{equation}\label{eq:laplace_expansion_corrected}
\begin{aligned}
\Delta\bigl( \M^{n} \times \Delta\e^{n+1} + \nabla\M^{n} \times \nabla\e^{n+1} \bigr)
= & \nabla\Delta\M^{n} \times \nabla\e^{n+1} 
+ \nabla^{2}\M^{n} \times \Delta\e^{n+1} 
 + \Delta\M^{n} \times \nabla^{2}\e^{n+1} \\
& + \Delta\M^{n} \times \Delta\e^{n+1} 
 + 3\nabla\M^{n} \times \nabla\Delta\e^{n+1} 
+ \M^{n} \times \Delta^{2}\e^{n+1}.
\end{aligned}
\end{equation}

\noindent \emph{Estimate of $J_1$}: Then, using expansion \eqref{eq:laplace_expansion_corrected} and applying H\"older's inequality term by term, we achieve
\begin{equation}
\begin{aligned}
J_1 &= \bigl(\nabla\Delta\M^{n}\times\nabla\e^{n+1} + \nabla^{2}\M^{n}\times\Delta\e^{n+1} 
      + \Delta\M^{n}\times\nabla^{2}\e^{n+1}  + \Delta\M^{n}\times\Delta\e^{n+1} 
     3\nabla\M^{n}\times\nabla\Delta\e^{n+1} + \M^{n}\times\Delta^{2}\e^{n+1}, \Delta^{2}\e^{n+1}\bigr) \\
    &\leq \Bigl( \|\nabla\Delta\M^{n}\|_{L^{2}} \|\nabla\e^{n+1}\|_{L^{\infty}} 
            + \|\nabla^{2}\M^{n}\|_{L^{4}} \|\Delta\e^{n+1}\|_{L^{4}} 
 + \|\Delta\M^{n}\|_{L^{4}} (\|\nabla^{2}\e^{n+1}\|_{L^{4}} + \|\Delta\e^{n+1}\|_{L^{4}}) \\
    &\qquad + 4\|\nabla\M^{n}\|_{L^{\infty}} \|\nabla\Delta\e^{n+1}\|_{L^{2}} \Bigr)
            \|\Delta^{2}\e^{n+1}\|_{L^{2}} \\
    &\leq \epsilon^{-1}\tau^{2} + \epsilon^{-1}\|\nabla\Delta\e^{n+1}\|_{L^{2}}^{2} 
            + \epsilon\|\Delta^{2}\e^{n+1}\|_{L^{2}}^{2}.
\end{aligned}
\end{equation}
where by mathematical induction assumption we have $||\nabla \Delta \e^n||_{L^2}=||\nabla \Delta( \m^n -\M^n)||_{L^2} \le C_0\tau \le C.$ Furthermore, applying the triangle inequality yields 
\begin{equation}\label{third_Mn}
    ||\nabla \Delta\M^n||_{L^2} \le C_0 \tau + ||\nabla \Delta\m^n||_{L^2} \le C_0 \tau +  T ||\partial_t  \m(t)||_{L^{\infty}(t_k , t_{k+1};H^3(\Omega))}   + ||\m_0||_{H^3} \le C .
\end{equation}
Besides, from Lemma \ref{lemma52} it follows that $||\Delta \M^n||_{L^4}, ||\nabla^2 \M^n||_{L^4} \le C$ and Gagliardo-Nirenberg inequality gives that $||\nabla \M^n||_{L^{\infty}}\le ||\M^n||_{W^{2,4}}.$ Combining this with the Sobolev embedding $H^2(\Omega) \hookrightarrow L^{\infty}(\Omega)$ 
and applying Theorem \ref{temporal_1}, we obtain
\begin{equation*}
    ||\nabla \e^{n+1}||_{L^{\infty}} \le C ||\nabla \e^{n+1}||_{H^2} \le C (||\nabla \e^{n+1}||_{L^2} + ||\nabla \Delta \e^{n+1}||_{L^2}) \le C (\tau + ||\nabla \Delta \e^{n+1}||_{L^2}).
\end{equation*}
Since we know that $||\nabla^2 \e^{n+1}||_{L^4} \le C ||\Delta \e^{n+1}||_{L^4}.$ By Sobolev embedding and the definition of $H^1$ norm, we have
\begin{equation}
    ||\Delta \e^{n+1}||_{L^4} \le C ||\Delta \e^{n+1}||_{H^1} \le C(||\Delta \e^{n+1}||_{L^2} + ||\nabla \Delta \e^{n+1}||_{L^2}).
\end{equation}

\noindent \emph{Estimate of $J_2$}: Second, applying the same operations as previously described, we obtain
\begin{equation}
    \begin{aligned}
    J_2={} &\bigl(\nabla\Delta\e^{n}\times\nabla\m^{n+1} + \nabla^{2}\e^{n}\times\Delta\m^{n+1} + \Delta\e^{n}\times\nabla^{2}\m^{n+1} + \nabla\e^{n}\times\nabla\Delta\m^{n+1},\, \Delta^{2}\e^{n+1}\bigr) \\
        &\quad + \bigl(\Delta\e^{n}\times\Delta\m^{n+1} + 2\nabla\e^{n}\times\nabla\Delta\m^{n+1}  + \e^{n}\times\Delta^{2}\m^{n+1},\, \Delta^{2}\e^{n+1}\bigr) \\
    \leq{} &\Bigl(\|\nabla\Delta\e^{n}\|_{L^{2}} \|\nabla\m^{n+1}\|_{L^{\infty}} 
            + \|\nabla^{2}\e^{n}\|_{L^{4}} \|\Delta\m^{n+1}\|_{L^{4}}  + \|\Delta\e^{n}\|_{L^{4}} \|\nabla^{2}\m^{n+1}\|_{L^{4}}\Bigr) 
            \|\Delta^{2}\e^{n+1}\|_{L^{2}} \\
        & + \Bigl(2\|\nabla\e^{n}\|_{L^{\infty}} \|\nabla\Delta\m^{n+1}\|_{L^{2}} 
            + \|\Delta\e^{n}\|_{L^{4}} \|\Delta\m^{n+1}\|_{L^{4}}\Bigr) 
            \|\Delta^{2}\e^{n+1}\|_{L^{2}}  + \|\e^{n}\|_{L^{\infty}} \|\Delta^{2}\m^{n+1}\|_{L^{2}} 
            \|\Delta^{2}\e^{n+1}\|_{L^{2}} \\
    \leq{} &\epsilon^{-1}\tau^{2} + \epsilon^{-1}\|\nabla\Delta\e^{n}\|_{L^{2}}^{2} 
            + \epsilon\|\Delta^{2}\e^{n+1}\|_{L^{2}}^{2},
    \end{aligned}
\end{equation}

The Sobolev embedding $H^{2} \hookrightarrow L^{\infty}$ (valid in dimensions $d \le 3$) gives $|| \e^{n}||_{L^{\infty}} \le C || \e^{n}||_{H^2}$. 
Together with Theorem \ref{temporal_1} and Lemma \ref{lemma_tau1} we arrive at
\begin{equation*}
    || \e^{n}||_{H^2} \le C (|| \e^{n}||_{L^2} + ||\Delta \e^{n}||_{L^2}) \le C \tau.
\end{equation*}
\noindent \emph{Estimate of $J_3$}: Finally, the third term $J_3$ splits into six parts:
\begin{equation}
    \begin{aligned}
        J_3 &= \alpha \bigl(\Delta (|\nabla \M^{n}|^2 \M^{n+1} - |\nabla \m^{n}|^2 \m^{n+1}), -\Delta^2 \e^{n+1}\bigr) \\
        &= J_{31} + J_{32} + J_{33} + J_{34} + J_{35} + J_{36},
    \end{aligned}
\end{equation}
where
\begin{align}
    J_{31} &= \bigl(2\nabla \Delta \M^n \nabla \M^n \e^{n+1} + \nabla^2 \M^n \Delta \M^n \e^{n+1} 
          + \nabla^2 \M^n \nabla \M^n \nabla \e^{n+1},\; \Delta^2 \e^{n+1}\bigr), \nonumber \\
    J_{32} &= \bigl(2\nabla \M^n \Delta \M^n \nabla \e^{n+1} + |\nabla \M^n|^2 \Delta \e^{n+1},\; \Delta^2 \e^{n+1}\bigr), \nonumber \\
    J_{33} &= \bigl(\nabla \Delta \e^n \nabla \M^{n} \m^{n+1} + \nabla^2 \e^n \Delta \M^n \m^{n+1} 
          + \nabla^2 \e^n \nabla \M^n \nabla \m^{n+1},\; \Delta^2 \e^{n+1}\bigr), \nonumber \\
    J_{34} &= \bigl(2\nabla \e^n \nabla^2 \e^n \nabla \m^{n+1} + |\nabla \e^n|^2 \Delta \m^{n+1},\; \Delta^2 \e^{n+1}\bigr), \nonumber \\
    J_{35} &= \bigl(\Delta \e^n \nabla^2 \m^n \m^{n+1} + \nabla \e^n \nabla \Delta \m^n \m^{n+1} 
          + \nabla \e^n \nabla^2 \m^n \nabla \m^{n+1},\; \Delta^2 \e^{n+1}\bigr), \nonumber \\
    J_{36} &= \bigl(\Delta \e^n \nabla \m^n \nabla \m^{n+1} + \nabla \e^n \Delta \m^n \nabla \m^{n+1} 
          + \nabla \e^n \nabla \m^n \Delta \m^{n+1},\; \Delta^2 \e^{n+1}\bigr). \nonumber
\end{align}
Now we estimate each term using the bound \eqref{third_Mn} and the Sobolev inequalities
\begin{equation}
    \|\nabla^2 \e^{n}\|_{L^4} \le C \|\Delta \e^{n}\|_{L^4}, \quad
    \|\nabla^2 \e^{n}\|_{L^2} \le C \|\Delta \e^{n}\|_{L^2},
    \label{sobolev_eq}
\end{equation}
which follow from \cite[Proposition 2.8]{katzourakis2019numerical}.

\noindent \emph{Estimate of $J_{31}$:}
\begin{equation}
    \begin{aligned}
        J_{31} &\le \bigl(\|\nabla \Delta \M^n\|_{L^2} \|\nabla \M^n\|_{L^{\infty}} \|\e^{n+1}\|_{L^{\infty}}  + \|\nabla^2 \M^n\|_{L^4} \|\Delta \M^n\|_{L^4} \|\e^{n+1}\|_{L^{\infty}} \\
        &\quad + \|\nabla^2 \M^n\|_{L^4} \|\nabla \M^n\|_{L^{\infty}} \|\nabla \e^{n+1}\|_{L^4}\bigr)
             \|\Delta^2 \e^{n+1}\|_{L^2}.
    \end{aligned}
\end{equation}

\noindent \emph{Estimate of $J_{32}$:}
\begin{equation}
    J_{32} \le \bigl(2\|\nabla \M^n\|_{L^{\infty}} \|\Delta \M^n\|_{L^4} \|\nabla \e^{n+1}\|_{L^4}
          + \|\nabla \M^n\|_{L^{\infty}}^2 \|\Delta \e^{n+1}\|_{L^2}\bigr)
          \|\Delta^2 \e^{n+1}\|_{L^2}.
\end{equation}

\noindent \emph{Estimate of $J_{33}$:}
\begin{equation}
    \begin{aligned}
        J_{33} &\le \bigl(\|\nabla \Delta \M^n\|_{L^2} \|\nabla \e^n\|_{L^{\infty}} \|\m^{n+1}\|_{L^{\infty}}  + \|\nabla^2 \M^n\|_{L^4} \|\Delta \e^n\|_{L^4} \|\m^{n+1}\|_{L^{\infty}} \\
        &\quad + \|\nabla^2 \M^n\|_{L^4} \|\nabla \e^n\|_{L^4} \|\nabla \m^n\|_{L^{\infty}}\bigr)
             \|\Delta^2 \e^{n+1}\|_{L^2}.
    \end{aligned}
\end{equation}

\noindent \emph{Estimate of $J_{34}$:}
\begin{equation}
    \begin{aligned}
        J_{34} &\le \bigl(\|\nabla \Delta \e^n\|_{L^2} \|\nabla \m^{n}\|_{L^{\infty}} \|\m^{n+1}\|_{L^{\infty}}  + \|\nabla^2 \e^n\|_{L^4} \|\Delta \m^n\|_{L^4} \|\m^{n+1}\|_{L^{\infty}} \bigr)
             \|\Delta^2 \e^{n+1}\|_{L^2}.
    \end{aligned}
\end{equation}

\noindent \emph{Estimate of $J_{35}$:}
\begin{equation}
    \begin{aligned}
        J_{35} &\le \bigl(\|\Delta \e^n\|_{L^{4}} \|\nabla^2 \m^n\|_{L^{\infty}} \| \m^{n+1}\|_{L^{\infty}} + \|\nabla \e^n\|_{L^{\infty}} \|\nabla \Delta \m^n\|_{L^2} \|\Delta \m^{n+1}\|_{L^{\infty}} \\
        &\quad + \|\Delta \e^n\|_{L^4} \|\nabla^2 \m^n\|_{L^4} \|\nabla \m^{n}\|_{L^{\infty}}\bigr)
             \|\Delta^2 \e^{n+1}\|_{L^2}.
    \end{aligned}
\end{equation}

\noindent \emph{Estimate of $J_{36}$:}
\begin{equation}
    \begin{aligned}
        J_{36} &\le \bigl( \|\Delta \e^n\|_{L^2} \|\nabla \m^n\|_{L^{\infty}} \|\nabla \m^{n+1}\|_{L^{\infty}} +\|\nabla \e^n\|_{L^4} \|\Delta \m^n\|_{L^4} \|\nabla \m^{n+1}\|_{L^{\infty}} \\
        &\quad + \|\nabla \e^n\|_{L^4} \|\nabla \m^n\|_{L^{\infty}} \|\Delta \m^{n+1}\|_{L^4}\bigr)
             \|\Delta^2 \e^{n+1}\|_{L^2}.
    \end{aligned}
\end{equation}

\noindent \emph{Final estimate:}
Summing all contributions and using the regularity assumptions on $\M^n$, $\m^n$, $\m^{n+1}$ from \eqref{third_Mn}, we obtain
\begin{equation}
    \begin{aligned}
        \sum_{i=1}^{6} J_{3i} &\le \eps^{-1} \tau^2 + \eps^{-1} \|\nabla \Delta \e^{n}\|_{L^2}^2 
                            + \eps \|\Delta^2 \e^{n+1}\|_{L^2}^2.
    \end{aligned}
\end{equation}
Therefore, using the bound \eqref{third_Mn} and the inequalities 
$\|\nabla^2 \e^{n}\|_{L^4} \le C \|\Delta \e^{n}\|_{L^4}$, 
$\|\nabla^2 \e^{n}\|_{L^2} \le C \|\Delta \e^{n}\|_{L^2}$ 
\cite[Proposition 2.8]{katzourakis2019numerical}, each term $J_i$ is estimated similarly.
Summing all contributions yields
\begin{equation}
    J_3 \le \eps^{-1} \tau^2 + \eps^{-1} \|\nabla \Delta \e^{n}\|_{L^2}^2 + \eps \|\Delta^2 \e^{n+1}\|_{L^2}^2 ,
\end{equation}
where the Sobolev embeddings $H^{1}(\Omega) \hookrightarrow L^{4}(\Omega)$, 
$H^{2}(\Omega) \hookrightarrow L^{\infty}(\Omega)$, together with the 
mathematical induction hypothesis, yield
\begin{align*}
    \|\Delta \e^{n}\|_{L^{4}} 
    &\leq C \|\Delta \e^{n}\|_{H^{1}} 
     = C\bigl(\|\Delta \e^{n}\|_{L^{2}} + \|\nabla \Delta \e^{n}\|_{L^{2}}\bigr), \\
    \|\nabla \e^{n}\|_{L^{\infty}} 
    &\leq C \|\nabla \e^{n}\|_{H^{2}} 
     \leq C\bigl(\|\nabla \e^{n}\|_{L^{2}} + \|\nabla \Delta \e^{n}\|_{L^{2}}\bigr), \\
    \|\nabla \e^{n}\|_{L^{4}} 
    &\leq C \|\nabla \e^{n}\|_{H^{1}} 
     \leq C\bigl(\|\nabla \e^{n}\|_{L^{2}} + \|\nabla^{2} \e^{n}\|_{L^{2}}\bigr) \leq C\bigl(\|\nabla \e^{n}\|_{L^{2}} + \|\Delta \e^{n}\|_{L^{2}}\bigr) 
     \leq C \tau.
\end{align*}
Follow the same estimates as in \eqref{mnaDel_m} we obtain 
\begin{equation}
    || \Delta^2 \m^n ||_{L^2} \le T   ||\partial_t  \m(t)||_{L^{\infty}(t_k , t_{k+1};H^4(\Omega))}   + ||\m_0||_{H^4} \le C .
\end{equation}
Combining all the above estimates and selecting a small positive value for $\eps$, we derive
\begin{equation*}
    \frac{1}{2} D_{\tau}||\nabla \Delta \e^{n+1}||_{L^2}^2 +  \frac{\alpha }{2} ||\Delta^2 \e^{n+1}||_{L^2}^2 \le \eps^{-1}  \tau^2  + \eps^{-1} ||\nabla \Delta \e^{n+1} ||_{L^2}^2 + \eps ||\Delta^2 \e^{n+1} ||_{L^2}^2 +  \eps^{-1} || \Delta R_{tr}^{n+1} ||_{L^2}^2. 
\end{equation*}
Applying the discrete version of Gronwall's inequality in Lemma \ref{Gronwall} leads to \eqref{litte_aim2}. This completes the proof of the lemma.
\end{proof}

In order to derive the optimal convergence rate under the framework of LOD, the boundedness of $||\Delta  D_{\tau} \M^{n+1}||_{H^1}$ is needed, and a similar result has been used in \cite[Conclusion 10.7]{henning2022superconvergence}. 
\subsubsection{Lemma \ref{lemma_tau3}}
\label{subsec:Proof of ref{lemma_tau3}.}

This lemma establishes the boundedness of $||\Delta  D_{\tau} \M^{n+1}||_{H^1}$, a result that will be employed in the proof of Theorem \ref{theorem_LODestimate} to derive the optimal convergence rate within the LOD framework.
\begin{lemma}\label{lemma_tau3}
Assume \eqref{assumption} holds, and considering $\M^{n+1} $ as the solution to the elliptic system \eqref{time_discrete_form}. Then, there exists a $\tau$-independent constant $C>0$ such that
    \begin{equation}
        ||\Delta  D_{\tau} \M^{n+1}||_{H^1} \le C .
    \end{equation}
\end{lemma}

\begin{proof}
Recall that $\e^{n} := \m(t_{n}) - \M^{n}$. By inserting the intermediate terms $\Delta\m(t_{n+1})$ and $\Delta\m(t_{n})$ and using the definition of the $H^1$ norm of $\Delta D_{\tau} \M^{n+1}$, we obtain
\begin{equation}
    \begin{aligned}
        \|\Delta D_{\tau} \M^{n+1}\|_{H^{1}} 
        &= \|\Delta D_{\tau} \M^{n+1}\|_{L^{2}} + \|\nabla(\Delta D_{\tau} \M^{n+1})\|_{L^{2}} \\
        &= \tau^{-1} \|\Delta\M^{n+1} - \Delta\m(t_{n+1}) + \Delta\m(t_{n+1}) - \Delta\m(t_{n}) + \Delta\m(t_{n}) - \Delta\M^{n}\|_{L^{2}} \\
        &\quad + \|D_{\tau}(\nabla\Delta\M^{n+1})\|_{L^{2}} \\
        &= \tau^{-1} \Bigl( \|{-}\Delta\e^{n+1} + \Delta\m(t_{n+1}) - \Delta\m(t_{n}) + \Delta\e^{n}\|_{L^{2}} \\
        &\quad + \|{-}\nabla\Delta\e^{n+1} + \nabla\Delta\m(t_{n+1}) - \nabla\Delta\m(t_{n}) + \nabla\Delta\e^{n}\|_{L^{2}} \Bigr).
    \end{aligned}
\end{equation}
Therefore, using Lemma \ref{lemma_tau1} and Lemma \ref{lemma_tau2}, we can bound $||D_{\tau} \Delta \e^{n+1} ||_{L^2}$ and $||D_{\tau} \nabla \Delta \e^{n+1} ||_{L^2}$, respectively. Furthermore, based on the regularity assumption in \eqref{assumption}, we derive
\begin{equation*}
    \tau^{-1}\|\nabla\Delta\m(t_{n+1}) - \nabla\Delta\m(t_{n})\|_{L^{2}}
    = \tau^{-1}\left\|\int_{t_{n}}^{t_{n+1}} \partial_{t}\nabla\Delta\m(t) \, dt\right\|_{L^{2}}
    \le \sup_{t} \|\partial_{t}\m\|_{H^{3}(\Omega)} \le C,
\end{equation*}
and similarly for $\tau^{-1}\|\Delta\m(t_{n+1}) - \Delta\m(t_{n})\|_{L^{2}}$.
\end{proof}

\subsection{Proof of Theorem \ref{theorem_LODestimate}. }
\label{subsec:Proof of Theorem ref{theorem_LODestimate}.}

This theorem estimates the difference between the LOD projection $A_{\LOD}(\M^n)$ and the discrete homogenized solution $\M_{\LOD}^{n}$. We proceed to prove \eqref{LODerror} by mathematical induction.

\begin{proof}
For the case $n = 0$, recalling that $\M^{0} = \m_{0} = \m(t_{0})$, we apply the triangle inequality to achieve 
 \begin{equation*}
 \begin{aligned}
     ||\e_{\LOD}^{0}||_{H^1} &= ||A_{\LOD}(\M^0) -\M_{\LOD}^{0}||_{H^1} \\
     &\le ||A_{\LOD}(\M^0) -\M^0||_{H^1} + ||\M^0 -\M_{\LOD}^{0}||_{H^1} \\
     &\le ||A_{\LOD}(\m_0) -\m_0||_{H^1} + ||\m_0 -\M_{\LOD}^{0}||_{H^1} \\
     \text{( C\'ea's lemma)} &\le ||A_{\LOD}(\m_0) -\m_0||_{H^1} + C ||\m_0 -A_{\LOD}(\m_0)||_{H^1} \\
     &\le CH^2||\f^0||_{H^1},
 \end{aligned}    
 \end{equation*}
 and 
 \begin{equation*}
 \begin{aligned}
     ||\e_{\LOD}^{0}||_{L^2} &\le ||A_{\LOD}(\m_0) -\m_0||_{L^2} + ||\m_0 -\M_{\LOD}^{0}||_{L^2}  \\
     &= ||A_{\LOD}(\m_0) -\m_0||_{L^2} + ||\m_0 -\M_{\LOD}^{0} - P_H (\m_0 -\M_{\LOD}^{0})||_{L^2} \\
     &\le CH^3||\f^0||_{H^1} + CH ||\m_0 -\M_{\LOD}^{0}||_{H^1}\\
     \text{( C\'ea's lemma)}&\le CH^3||\f^0||_{H^1} + CH ||\m_0 -A_{\LOD}(\m_0)||_{H^1} \\
     &\le CH^3||\f^0||_{H^1},
 \end{aligned}    
 \end{equation*}
 with $\f^0:=-\alpha \Delta \m_0+ \nabla \cdot (\m_0 \times \nabla \m_0) + \m_0$ and from the assumption $\m_0 \in \H^4(\Omega)$ as stated in \eqref{assumption} it follows that $\f^0 \in \H^1(\Omega).$ Next, we want to prove \eqref{LODerror} by mathematical induction. First, let us assume that \eqref{LODerror} is valid for $k = 1,\cdots, n$, i.e.
\begin{equation}\label{spaceError}
    || \e_{\LOD}^k||_{L^2} + H || \e_{\LOD}^k||_{H^1}  \le C_0 H^3 .
\end{equation}
Then, we will prove that \eqref{LODerror} holds for $k = 1,\cdots, n+1$.
Since 
\begin{equation}
	\begin{aligned}
		(D_{\tau} \M^{n+1},\v) + \alpha(  \nabla \M^{n+1},\nabla \v)  - ( \M^{n} \times \nabla  \M^{n+1} , \nabla \v) 
		=\alpha ( (\nabla  \M^{n} \cdot \nabla \M^{n} )\M^{n+1}, \v) ,
	\end{aligned}
\end{equation}
for all $\v \in \V_{\LOD}$. Recalling that $B^n(A_{\LOD}(u),v) = B^n(u,v)$ for $v \in V_{\LOD}$, we arrive at
\begin{equation}\label{438}
	\begin{aligned}
		&(D_{\tau} \M^{n+1},\v) + \alpha(  \nabla A_{\LOD}(\M^{n+1}),\nabla \v)  - ( \M^{n} \times \nabla  A_{\LOD}(\M^{n+1} ), \nabla \v) + (  A_{\LOD}(\M^{n+1}),\v) \\
		=&\alpha ( (\nabla  \M^{n} \cdot \nabla \M^{n} )\M^{n+1}, \v) + ( \M^{n+1},\v) .
	\end{aligned}
\end{equation}

Subtracting equation \eqref{438} from \eqref{436}, we obtain the spatial error equation
\begin{equation}
    \begin{aligned}
        &(D_{\tau} \e_{\LOD}^{n+1},\v) + \alpha(  \nabla \e_{\LOD}^{n+1},\nabla \v)  
        - ( \M^{n} \times \nabla A_{\LOD}( \M^{n+1} ) - \M_{\LOD}^{n} \times \nabla \M_{\LOD}^{n+1}, \nabla \v) \\
        =&\alpha \big( (\nabla  \M^{n} \cdot \nabla \M^{n} )\M^{n+1} 
          - (\nabla \M_{\LOD}^{n} \cdot \nabla \M_{\LOD}^{n} )\M_{\LOD}^{n+1}, \v \big) \\
        &- \big( D_{\tau} (\M^{n+1} - A_{\LOD} (\M^{n+1}) ),\v \big) 
          + \big( \M^{n+1} - A_{\LOD} (\M^{n+1}) ,\v \big),
    \end{aligned}
\end{equation}
for all $\v \in V_{\LOD}$. Then, testing with $\v = \e_{\LOD}^{n+1}$, we obtain
\begin{equation}
	\begin{aligned}
		&\frac{1}{2}D_{\tau}( ||\e_{\LOD}^{n+1}||_{L^2}^2 ) + \alpha ||\nabla \e_{\LOD}^{n+1}||_{L^2}^2  \\
		=& \underbrace{( \M^{n} \times \nabla A_{\LOD}( \M^{n+1} ) - \M_{\LOD}^{n} \times \nabla \M_{\LOD}^{n+1}, \nabla \e_{\LOD}^{n+1}) }_{:=S_1} + \underbrace{\alpha ( |\nabla  \M^{n}|^2 \M^{n+1} - |\nabla \M_{\LOD}^{n}|^2 \M_{\LOD}^{n+1}, \e_{\LOD}^{n+1})}_{:=S_2}  \\
         &- \underbrace{(D_{\tau} (\M^{n+1} - A_{\LOD} (\M^{n+1}) ),\e_{\LOD}^{n+1})}_{:=S_3} + \underbrace{(\M^{n+1} - A_{\LOD} (\M^{n+1}) ,\e_{\LOD}^{n+1})}_{:=S_4}. \\ 
    \end{aligned}
\end{equation}
\noindent \emph{Estimate of $S_{1}$:} We split it into 
\begin{equation*}
	\begin{aligned}
		S_1 =& ( \M^{n} \times \nabla A_{\LOD}( \M^{n+1} ) - \M_{\LOD}^{n} \times \nabla \M_{\LOD}^{n+1}, \nabla \e_{\LOD}^{n+1}) \\
          =& \underbrace{( (\M^{n}- \M_{\LOD}^{n}) \times \nabla (A_{\LOD}( \M^{n+1} ) -\M^{n+1} ) , \nabla \e_{\LOD}^{n+1}) }_{:=S_{11}} \\
          &+ \underbrace{((\M^{n}- \M_{\LOD}^{n}) \times \nabla \M^{n+1}, \nabla \e_{\LOD}^{n+1}) }_{:=S_{12}} 
           + \underbrace{( \M_{\LOD}^{n} \times \nabla \e_{\LOD}^{n+1}, \nabla \e_{\LOD}^{n+1}) }_{=0} .
          \end{aligned}
\end{equation*}
Applying H\"{o}lder's inequality and the Sobolev embedding $H^{1}(\Omega) \hookrightarrow L^{4}(\Omega)$, 
the term $S_{11}$ can be bounded as follows
\begin{equation*}
    \begin{aligned}
        S_{11} &\leq \|\M^{n} - A_{\LOD}(\M^{n}) + A_{\LOD}(\M^{n}) - \M_{\LOD}^{n}\|_{L^{4}} \\
        &\qquad \times \|\nabla(A_{\LOD}(\M^{n+1}) - \M^{n+1})\|_{L^{2}} \|\nabla\e_{\LOD}^{n+1}\|_{L^{4}} \\
        &\leq C H^{2} \bigl(\|\M^{n} - A_{\LOD}(\M^{n})\|_{L^{4}} + \|\e_{\LOD}^{n}\|_{L^{4}}\bigr) \|\nabla\e_{\LOD}^{n+1}\|_{H^{1}} \\
        &\leq C H^{2} \bigl(\|\M^{n} - A_{\LOD}(\M^{n})\|_{H^{1}} + \|\e_{\LOD}^{n}\|_{H^{1}}\bigr) \|\nabla\e_{\LOD}^{n+1}\|_{H^{1}},
    \end{aligned}
\end{equation*}
where we used Theorem $\ref{theorem_Projectionestimate}$ to derive $||\M^{n} -A_{\LOD}( \M^{n} ) ||_{H^1} \le CH^2$ and $||\nabla (A_{\LOD}( \M^{n+1} ) - \M^{n+1}) ||_{L^2} \le CH^2 .$ Then, from inverse estimate in Lemma \ref{LOD_inverse_estimate} it follows that $||\nabla \e_{\LOD}^{n+1}||_{H^1} \le C H^{-1} ||\nabla \e_{\LOD}^{n+1}||_{L^2}$ and $||\e_{\LOD}^{n}||_{H^1} \le C H^{-1} ||\e_{\LOD}^{n}||_{L^2} $. Together with Young's inequality, we arrive at

\begin{equation*}
    S_{11} \le C H^3 \|\nabla \e_{\LOD}^{n+1}\|_{L^2} 
         + C \|\e_{\LOD}^{n}\|_{L^2} \|\nabla \e_{\LOD}^{n+1}\|_{L^2}
         \le C \epsilon^{-1} H^6 + \epsilon \|\nabla \e_{\LOD}^{n+1}\|_{L^2}^2 
         + C \epsilon^{-1} \|\e_{\LOD}^{n}\|_{L^2}^2.
\end{equation*}
For $S_{12}$ we again use H\"{o}lder's inequality together with estimate $||\M^{n} -A_{\LOD}( \M^{n} ) ||_{L^2} \le CH^3$ (Theorem $\ref{theorem_Projectionestimate}$) to conclude
\begin{equation*}
    \begin{aligned}
        S_{12} \leq{} & \|\M^{n} - A_{\LOD}(\M^{n}) + A_{\LOD}(\M^{n}) - \M_{\LOD}^{n}\|_{L^{2}} \|\nabla\e_{\LOD}^{n+1}\|_{L^{2}} \|\nabla\M^{n+1}\|_{L^{\infty}} \\
        \leq{} & \bigl(\|\M^{n} - A_{\LOD}(\M^{n})\|_{L^{2}} + \|\e_{\LOD}^{n}\|_{L^{2}}\bigr) \|\nabla\e_{\LOD}^{n+1}\|_{L^{2}} \\
        \leq{} & C H^{3} \|\nabla\e_{\LOD}^{n+1}\|_{L^{2}} 
                + C \|\e_{\LOD}^{n}\|_{L^{2}} \|\nabla\e_{\LOD}^{n+1}\|_{L^{2}} . 
    \end{aligned}
\end{equation*}
Consequently, using Young's inequality gives $S_1 \le C \eps^{-1} H^6 +  \eps ||\nabla \e_{\LOD}^{n+1} ||_{L^2}^2 + C \eps^{-1} ||\e_{\LOD}^{n} ||_{L^2}^2.$

\noindent \emph{Estimate of $S_2$:} $S_2$ can be rewritten as
\begin{equation}
    \begin{aligned}
        S_2 
        &= \alpha \big( |\nabla \M^{n}|^2 \M^{n+1} - |\nabla \M_{\LOD}^{n}|^2 \M_{\LOD}^{n+1}, \e_{\LOD}^{n+1} \big) \\ 
        &= \alpha \underbrace{\big( |\nabla \M_{\LOD}^{n}|^2 ( A_{\LOD}(\M^{n+1}) - \M_{\LOD}^{n+1} ), \e_{\LOD}^{n+1} \big)}_{S_{21}}  - \alpha \underbrace{\big( (|\nabla \M_{\LOD}^{n}|^2 - |\nabla \M^{n}|^2) A_{\LOD}(\M^{n+1}), \e_{\LOD}^{n+1} \big)}_{S_{22}} \\ 
        &\quad - \alpha \underbrace{\big( |\nabla \M^{n}|^2 (A_{\LOD}(\M^{n+1}) - \M^{n+1}), \e_{\LOD}^{n+1} \big)}_{S_{23}} ,
    \end{aligned}
\end{equation}

For terms $S_{21}$ and $S_{23}$, we use H\"{o}lder's inequality once again to obtain
\begin{equation}
    S_{21} \le \alpha \|\nabla \M_{\LOD}^{n}\|_{L^{\infty}}^2 \|\e_{\LOD}^{n+1}\|_{L^2}^2,
\end{equation}
\begin{equation}
    S_{23} \le \alpha \| \nabla \M^{n} \|_{L^{\infty}}^2 
           \|A_{\LOD}(\M^{n+1}) - \M^{n+1}\|_{L^2} \|\e_{\LOD}^{n+1}\|_{L^2}.
\end{equation}
Then, using Agmon's inequality and inverse estimate, we conclude that $||\nabla \M_{\LOD}^{n}||_{L^{\infty}} \le C$. Recalling that $\e_{\LOD}^{n} = A_{\LOD}(\M^n) - \M_{\LOD}^{n}$ and requiring $C_0 H^{\frac{1}{2}} \le C$, we apply the triangle inequality to obtain
\begin{equation}
\begin{aligned}
    ||\nabla \M_{\LOD}^{n}||_{L^{\infty}} &\le ||\nabla \e_{\LOD}^{n}||_{L^{\infty}} + ||\nabla A_{\LOD}(\M^n)||_{L^{\infty}} \\
    \text{(Agmon's~inequality)}&\le C ||\nabla \e_{\LOD}^{n}||_{H^1}^{\frac{1}{2}} ||\nabla \e_{\LOD}^{n}||_{H^2}^{\frac{1}{2}} + || \M^n||_{W^{1,\infty}} \\
    \text{(Inverse~estimate)}&\le H^{-\frac{3}{2}}||\nabla \e_{\LOD}^{n}||_{L^2} + || \M^n||_{W^{2,4}} \\
    \text{(Induction~hypothesis)}&\le C_0 H^{\frac{1}{2}} + || \M^n||_{W^{2,4}} \\
    &\le C . \\
\end{aligned}
\end{equation}
$S_{21}$ can be further bounded by $S_{21} \le C \|\e_{\LOD}^{n+1}\|_{L^2}^2.$ Besides, from Theorem \ref{theorem_Projectionestimate}, we have $\|A_{\LOD}(\M^{n+1}) - \M^{n+1}\|_{L^2} \le C H^3$ and from $H^2$ stability of $A_{\LOD}$ in Lemma \ref{LOD_inverse_estimate} it follows that $||A_{\LOD} (\M^n) ||_{L^{\infty}} \le ||A_{\LOD} (\M^n) ||_{H^{2}} \le C ||\M^n ||_{H^{2}}$. $S_{23}$ can be further bounded by $S_{23} \le C H^3 \|\e_{\LOD}^{n+1}\|_{L^2}.$

Noting the inequality that $\lvert \mathbf{a} +\mathbf{b} \rvert^2 \le 2\lvert \mathbf{a} \rvert^2 + 2\lvert \mathbf{b} \rvert^2$, 
and denoting $\mathbf{a} = \nabla \M_{\LOD}^{n} -\nabla A_{\LOD}(\M^{n})$, 
$\mathbf{b} = \nabla \bigl( A_{\LOD}(\M^{n}) - \M^{n} \bigr)$, 
we have
\begin{equation}
    \begin{aligned}
        S_{22} =& \Bigl( \bigl( \lvert \nabla \M_{\LOD}^{n} \rvert^2 - \lvert \nabla \M^{n} \rvert^2 \bigr) A_{\LOD}(\M^{n+1}), \, \e_{\LOD}^{n+1} \Bigr) \\
        =& \Bigl( \lvert \nabla \M_{\LOD}^{n} - \nabla \M^{n} \rvert^2 A_{\LOD}(\M^{n+1}), \, \e_{\LOD}^{n+1} \Bigr) + 2 \Bigl( \nabla \M^{n} \cdot \bigl( \nabla \M_{\LOD}^{n} - \nabla \M^{n} \bigr) A_{\LOD}(\M^{n+1}), \, \e_{\LOD}^{n+1} \Bigr) \\
        \le & 2 \bigl\lvert \bigl( \lvert \nabla \M_{\LOD}^{n} - \nabla A_{\LOD}(\M^{n}) \rvert^2 A_{\LOD}(\M^{n+1}), \, \e_{\LOD}^{n+1} \bigr) \bigr\rvert + 2 \bigl\lvert \bigl( \lvert \nabla \bigl( A_{\LOD}(\M^{n}) - \M^{n} \bigr) \rvert^2 A_{\LOD}(\M^{n+1}), \, \e_{\LOD}^{n+1} \bigr) \bigr\rvert \\
        &+ 2 \bigl\lvert \bigl( \nabla \M^{n} \cdot \nabla (\M_{\LOD}^{n} - A_{\LOD}(\M^{n}) ) A_{\LOD}(\M^{n+1}), \, \e_{\LOD}^{n+1} \bigr) \bigr\rvert \\
        &+ 2 \bigl\lvert \bigl( \nabla \M^{n} \cdot \nabla (A_{\LOD}(\M^{n}) -\M^{n}) A_{\LOD}(\M^{n+1}), \, \e_{\LOD}^{n+1} \bigr) \bigr\rvert.
    \end{aligned}
\end{equation}

Moreover, given that $\e_{\LOD}^{n}=A_{\LOD}(\M^n)-\M_{\LOD}^{n}$ 
and the Neumann boundary condition $\nabla \mathbf{M}^n \cdot \mathbf{n} = \mathbf{0}$ on $\partial \Omega$, an application of integration by parts to the final term of the above equation yields
\begin{equation*}
    \begin{aligned}
        &\bigl( \nabla (A_{\LOD}(\mathbf{M}^{n}) - \mathbf{M}^n), 
        (A_{\LOD}(\mathbf{M}^{n+1}) \cdot \mathbf{e}_{\LOD}^{n+1}) \nabla \mathbf{M}^{n} \bigr) \\
        =& -\bigl( (A_{\LOD}(\mathbf{M}^{n}) - \mathbf{M}^n), 
        \nabla \cdot \bigl( (A_{\LOD}(\mathbf{M}^{n+1}) \cdot \mathbf{e}_{\LOD}^{n+1}) \nabla \mathbf{M}^{n} \bigr) \bigr) + \underbrace{\int_{\partial \Omega} \bigl( (A_{\LOD}(\mathbf{M}^{n+1}) \cdot \mathbf{e}_{\LOD}^{n+1}) 
        \nabla \mathbf{M}^{n} \cdot \mathbf{n} \bigr) 
        (A_{\LOD}(\mathbf{M}^{n}) - \mathbf{M}^n) \, \d \S}_{=0}.
    \end{aligned}
\end{equation*}

By applying H\"{o}lder's inequality to $S_{22}$, we obtain
\begin{equation}
    \begin{aligned}
        & \bigl( (|\nabla\M_{\LOD}^{n}|^2 - |\nabla \mathbf{M}^{n}|^2) A_{\LOD}(\mathbf{M}^{n+1}), \mathbf{e}_{\LOD}^{n+1} \bigr) \\
        \le & 2\|\nabla \mathbf{e}_{\LOD}^{n}\|_{L^2} \|\nabla \mathbf{e}_{\LOD}^{n}\|_{L^4} \|A_{\LOD}(\mathbf{M}^{n+1})\|_{L^{\infty}} \|\mathbf{e}_{\LOD}^{n+1}\|_{L^4}  + 2\|\nabla (A_{\LOD}(\mathbf{M}^{n}) - \mathbf{M}^{n})\|_{L^2}^2 \|A_{\LOD}(\mathbf{M}^{n+1})\|_{L^{\infty}} \|\mathbf{e}_{\LOD}^{n+1}\|_{L^{\infty}} \\
        & + \|\nabla \mathbf{M}^{n}\|_{L^{\infty}} \|\nabla \mathbf{e}_{\LOD}^{n}\|_{L^2} \|A_{\LOD}(\mathbf{M}^{n+1})\|_{L^{\infty}} \|\mathbf{e}_{\LOD}^{n+1}\|_{L^2} + \|\Delta \mathbf{M}^{n}\|_{L^4} \|A_{\LOD}(\mathbf{M}^{n}) - \mathbf{M}^{n}\|_{L^2} \|A_{\LOD}(\mathbf{M}^{n+1})\|_{L^{\infty}} \|\mathbf{e}_{\LOD}^{n+1}\|_{L^4} \\
        & + \|\nabla \mathbf{M}^{n}\|_{L^{\infty}} \|\nabla A_{\LOD}(\mathbf{M}^{n+1})\|_{L^4} \|A_{\LOD}(\mathbf{M}^{n}) - \mathbf{M}^{n}\|_{L^2} \|\mathbf{e}_{\LOD}^{n+1}\|_{L^4} \\
        & + \|\nabla \mathbf{M}^{n}\|_{L^{\infty}} \|\nabla \mathbf{e}_{\LOD}^{n+1}\|_{L^2} \|A_{\LOD}(\mathbf{M}^{n+1})\|_{L^{\infty}} \|A_{\LOD}(\mathbf{M}^{n}) - \mathbf{M}^{n}\|_{L^2}.
    \end{aligned}
\end{equation}

Next, applying Lemma \ref{LOD_inverse_estimate} yields the uniform bound
\begin{equation*}
    \|A_{\LOD}(\mathbf{M}^{n+1})\|_{L^{\infty}} \le \|A_{\LOD}(\mathbf{M}^{n+1})\|_{H^2} 
    \le \|\mathbf{M}^{n+1}\|_{H^2} \le C.
\end{equation*}
Combining this estimate with Young's inequality, we further obtain
\begin{equation}
    \begin{aligned}
    S_{22}\le & \|\nabla \mathbf{e}_{\LOD}^{n}\|_{L^2} \|\nabla \mathbf{e}_{\LOD}^{n}\|_{H^1} \|\mathbf{e}_{\LOD}^{n+1}\|_{H^1} 
        + CH^4 \|\mathbf{e}_{\LOD}^{n+1}\|_{H^2}  + \|\nabla \mathbf{M}^{n}\|_{L^{\infty}} \|\nabla \mathbf{e}_{\LOD}^{n}\|_{L^2} 
        \|A_{\LOD}(\mathbf{M}^{n+1})\|_{L^{\infty}} \|\mathbf{e}_{\LOD}^{n+1}\|_{L^2} \\
        & + \|\Delta \mathbf{M}^{n}\|_{L^4} \|A_{\LOD}(\mathbf{M}^{n}) - \mathbf{M}^{n}\|_{L^2} 
        \|A_{\LOD}(\mathbf{M}^{n+1})\|_{L^{\infty}} \|\mathbf{e}_{\LOD}^{n+1}\|_{L^4} \\
        & + \|\nabla \mathbf{M}^{n}\|_{L^{\infty}} \|\nabla A_{\LOD}(\mathbf{M}^{n+1})\|_{L^4} 
        \|A_{\LOD}(\mathbf{M}^{n}) - \mathbf{M}^{n}\|_{L^2} \|\mathbf{e}_{\LOD}^{n+1}\|_{L^4} \\
        & + \|\nabla \mathbf{M}^{n}\|_{L^{\infty}} \|\nabla \mathbf{e}_{\LOD}^{n+1}\|_{L^2} 
        \|A_{\LOD}(\mathbf{M}^{n+1})\|_{L^{\infty}} \|A_{\LOD}(\mathbf{M}^{n}) - \mathbf{M}^{n}\|_{L^2} \\
        \le & \epsilon \|\nabla \mathbf{e}_{\LOD}^{n}\|_{L^2}^2 + \epsilon^{-1} \|\mathbf{e}_{\LOD}^{n+1}\|_{L^2}^2 
        + \epsilon^{-1} H^6 + \epsilon \|\mathbf{e}_{\LOD}^{n+1}\|_{H^1}^2 + \epsilon \|\nabla \mathbf{e}_{\LOD}^{n+1}\|_{L^2}^2,
    \end{aligned}
\end{equation}
where we have utilized the Sobolev embedding $H^{1} \hookrightarrow L^{4}$.

On one side, with mathematical induction assumption we have $||\nabla  \e_{\LOD}^{n}||_{L^2} \le C_0 H^2$ together with inverse estimate in Lemma \ref{LOD_inverse_estimate} yields 
\begin{equation*}
     ||\nabla \e_{\LOD}^{n}||_{H^1} 
     \le C H^{-1}  ||\nabla \e_{\LOD}^{n}||_{L^2} \le C_0 H \le \eps,
\end{equation*}
\begin{equation*}
     ||\nabla \e_{\LOD}^{n}||_{H^2} 
     \le C H^{-1}  ||\nabla \e_{\LOD}^{n}||_{H^1} ,
\end{equation*}
combined with the Sobolev embedding, we obtain
\begin{equation*}
     ||\e_{\LOD}^{n+1}||_{L^{\infty}} \le C ||\e_{\LOD}^{n+1}||_{H^2} \le C H^{-1} ||\e_{\LOD}^{n+1}||_{H^1} .
\end{equation*}
Hence, by Young's equality, we conclude that
\begin{equation*}
    S_2 \le C ||\e_{\LOD}^{n+1}||_{L^2}^2  + C \eps^{-1} H^6.
\end{equation*}
\noindent \emph{Estimate of $S_3$:} Applying H\"{o}lder's inequality and equation \eqref{another_elliptic_H3} in Theorem \ref{theorem_Projectionestimate} yields 

\begin{equation}
    \begin{aligned}
        S_3 &\leq \|D_{\tau}(\M^{n+1} - A_{\LOD}(\M^{n+1}))\|_{L^{2}} \|\e_{\LOD}^{n+1}\|_{L^{2}}
                 = \|D_{\tau}\M^{n+1} - A_{\LOD}(D_{\tau}\M^{n+1})\|_{L^{2}} \|\e_{\LOD}^{n+1}\|_{L^{2}} \\
               &\leq C H^{3} \|{-}\Delta D_{\tau}\M^{n+1} + \M^{n} \times \Delta D_{\tau}\M^{n+1}
                 + \nabla\M^{n} \times \nabla D_{\tau}\M^{n+1} + D_{\tau}\M^{n+1}\|_{H^{1}} \|\e_{\LOD}^{n+1}\|_{L^{2}}.
    \end{aligned}
\end{equation}
Since Lemma \ref{lemma52} and Lemma \ref{lemma_tau3} give that $|| D_{\tau} \M^{n+1}||_{H^1}, ||\Delta  D_{\tau} \M^{n+1}||_{H^1}\le C .$
Next, we estimate $\|\M^{n} \times \Delta D_{\tau}\M^{n+1} + \nabla \M^{n} \times \nabla D_{\tau}\M^{n+1}\|_{H^{1}}$.
Using the definition of the $H^{1}$-norm and Leibniz's rule, we obtain
\begin{equation}
    \begin{aligned}
        &\|\M^{n} \times \Delta D_{\tau}\M^{n+1} + \nabla \M^{n} \times \nabla D_{\tau}\M^{n+1}\|_{H^{1}} \\
        \leq{} &\|\M^{n} \times \Delta D_{\tau}\M^{n+1}\|_{L^{2}} 
                + \|\nabla(\M^{n} \times \Delta D_{\tau}\M^{n+1})\|_{L^{2}}  + \|\nabla \M^{n} \times \nabla D_{\tau}\M^{n+1}\|_{L^{2}} 
                + \|\nabla(\nabla \M^{n} \times \nabla D_{\tau}\M^{n+1})\|_{L^{2}} \\
        \leq{} &\|\nabla D_{\tau}\M^{n+1}\|_{L^{2}} 
                + \|\nabla^{2}\M^{n} \times \nabla D_{\tau}\M^{n+1}\|_{L^{2}}  + \|\nabla \M^{n} \times \nabla^{2} D_{\tau}\M^{n+1}\|_{L^{2}}  + \|\Delta D_{\tau}\M^{n+1}\|_{L^{2}} 
                + \|\nabla \M^{n} \times \Delta D_{\tau}\M^{n+1}\|_{L^{2}} \\
        &\quad + \|\nabla \M^{n} \times \nabla \Delta D_{\tau}\M^{n+1}\|_{L^{2}} \\
        \leq{} &\|\nabla D_{\tau}\M^{n+1}\|_{L^{2}} 
                + \|\nabla D_{\tau}\M^{n+1}\|_{L^{4}}  + \|\Delta D_{\tau}\M^{n+1}\|_{L^{2}} 
                + \|\nabla \Delta D_{\tau}\M^{n+1}\|_{L^{2}} \\
        \leq{} &C \|D_{\tau}\M^{n+1}\|_{H^{1}} 
                + C \|\Delta D_{\tau}\M^{n+1}\|_{H^{1}},
    \end{aligned}
\end{equation}
where we use the Sobolev embedding $H^{1}(\Omega) \hookrightarrow L^{4}(\Omega)$ to obtain
\[
\|\nabla D_{\tau}\M^{n+1}\|_{L^{4}} \leq C \|\nabla D_{\tau}\M^{n+1}\|_{H^{1}}.
\] Together with the definition of the $H^{1}$-norm of $\nabla D_{\tau}\M^{n+1}$, we conclude that
\begin{equation*}
    \|\nabla D_{\tau}\M^{n+1}\|_{H^{1}} 
    \leq \|\nabla D_{\tau}\M^{n+1}\|_{L^{2}} 
    + \|\nabla^{2} D_{\tau}\M^{n+1}\|_{L^{2}} 
    \leq \|D_{\tau}\M^{n+1}\|_{H^{1}} 
    + \|\Delta D_{\tau}\M^{n+1}\|_{L^{2}}.
\end{equation*}
Hence, by Young's equality, we deduce
\begin{equation*}
    S_3 \le  C \eps^{-1} H^6 + \eps || \e_{\LOD}^{n+1} ||_{L^2}^2.
\end{equation*}
\noindent \emph{Estimate of $S_4$:} Using Young's inequality and Theorem \ref{theorem_Projectionestimate}, we derive
\begin{equation*}
    S_4 = (\M^{n+1} - A_{\LOD}(\M^{n+1}),\, \e_{\LOD}^{n+1}) 
    \leq C \epsilon^{-1} H^{6} + \epsilon \|\e_{\LOD}^{n+1}\|_{L^{2}}^{2}.
\end{equation*}

By combining $S_1$-$S_4$, and employing discrete Gronwall's inequality, we ultimately obtain
\begin{equation}
\begin{aligned}
    ||\e_{\LOD}^{n+1}||_{L^2}^2 + \tau \sum \limits _{k=0} ^{n} || \nabla \e_{\LOD}^{k+1}||_{L^2}^2 \le C \exp(C_0) H^6.
\end{aligned}  
\end{equation}
Besides, $H^1$ error estimate can be obtained by inverse estimate in Lemma \ref{LOD_inverse_estimate}
\begin{equation*}
    ||\e_{\LOD}^{n+1}||_{H^1} \le C H^{-1} ||\e_{\LOD}^{n+1}||_{L^2}.
\end{equation*}

\end{proof}

\section*{Acknowledgments}
The authors acknowledge Professor Rong An, Professor Chengjie Liu, and Doctor Panchi Li for their valuable discussions and insights.

\section*{Funding}
ZM, RD, and LZ were partially supported by the National Natural Science Foundation of China (Grant No. 12271360). LZ was also partially supported by the Fundamental Research Funds for the Central Universities. RD was also partially supported by the Jiangsu Provincial Scientific Research Center of Applied Mathematics under Grant No. BK20233002. 

\bibliographystyle{IMANUM-BIB}
\bibliography{references}

@article{nirenberg1966extended,
  title={An extended interpolation inequality},
  author={Nirenberg, Louis},
  journal={Annali della Scuola Normale Superiore di Pisa-Scienze Fisiche e Matematiche},
  volume={20},
  number={4},
  pages={733--737},
  year={1966}
}

@book{agmon2010lectures,
  title     = {Lectures on elliptic boundary value problems},
  author    = {Agmon, Shmuel},
  volume    = {369},
  year      = {2010},
  publisher = {American Mathematical Society},
  address   = {Providence, Rhode Island},
  series    = {AMS Chelsea Publishing},
  note      = {Revised edition of the 1965 original}
}

@article{henning2017crank,
  title="{Crank--Nicolson Galerkin approximations to nonlinear Schr{\"o}dinger equations with rough potentials}",
  author={Henning, Patrick and Peterseim, Daniel},
  journal={Mathematical Models and Methods in Applied Sciences},
  volume={27},
  number={11},
  pages={2147--2184},
  year={2017},
  publisher={World Scientific}
}

@article{liu1979spin,
  title={Spin waves in static non-periodic magnetic structures},
  author={Liu, SH},
  journal={Journal of Magnetism and Magnetic Materials},
  volume={12},
  number={3},
  pages={262--276},
  year={1979},
  publisher={Elsevier}
}

@article{gutfleisch2011magnetic,
  title={Magnetic materials and devices for the 21st century: stronger, lighter, and more energy efficient},
  author={Gutfleisch, Oliver and Willard, Matthew A and Br{\"u}ck, Ekkes and Chen, Christina H and Sankar, SG and Liu, J Ping},
  journal={Advanced Materials},
  volume={23},
  number={7},
  pages={821--842},
  year={2011},
  publisher={Wiley Online Library}
}

@article{ee03,
  title={The heterognous multiscale methods},
  author={E, Weinan and Engquist, Bjorn},
  journal={Communications in Mathematical Sciences},
  volume={1},
  number={1},
  pages={87--132},
  year={2003},
  publisher={International Press of Boston}
}

@article{abdulle2014analysis,
	Author = {Abdulle, Assyr and Vilmart, Gilles},
	Journal = {Mathematics of Computation},
	Number = {286},
	Pages = {513--536},
	Title = {Analysis of the finite element heterogeneous multiscale method for quasilinear elliptic homogenization problems},
	Volume = {83},
	Year = {2014}}

@article{ming2005analysis,
	Author = {E, Weinan and Ming, Pingbing and Zhang, Pingwen},
	Journal = {Journal of the American Mathematical Society},
	Number = {1},
	Pages = {121--156},
	Title = {Analysis of the heterogeneous multiscale method for elliptic homogenization problems},
	Volume = {18},
	Year = {2005}}

@article{berlyand2010flux,
	Author = {Berlyand, Leonid and Owhadi, Houman},
	Journal = {Archive for Rational Mechanics and Analysis},
	Number = {2},
	Pages = {677--721},
	Publisher = {Springer},
	Title = {Flux norm approach to finite dimensional homogenization approximations with non-separated scales and high contrast},
	Volume = {198},
	Year = {2010}}

@article{Owhadi:2011,
  title={Localized bases for finite-dimensional homogenization approximations with nonseparated scales and high contrast},
  author={Owhadi, Houman and Zhang, Lei},
  journal={Multiscale Modeling \& Simulation},
  volume={9},
  number={4},
  pages={1373--1398},
  year={2011},
  publisher={SIAM}
}

@article{Arbogast_two_scale_04,
  title={Analysis of a two-scale, locally conservative subgrid upscaling for elliptic problems},
  author={Arbogast, Todd},
  journal={SIAM Journal on Numerical Analysis},
  volume={42},
  number={2},
  pages={576--598},
  year={2004},
  publisher={SIAM}
}

@article{egw10,
  title={Multiscale finite element methods for high-contrast problems using local spectral basis functions},
  author={Efendiev, Yalchin and Galvis, Juan and Wu, Xiao-Hui},
  journal={Journal of Computational Physics},
  volume={230},
  number={4},
  pages={937--955},
  year={2011},
  publisher={Elsevier}
}

@book{eh09,
  title={Multiscale finite element methods: theory and applications},
  author={Efendiev, Yalchin and Hou, Thomas Y},
  volume={4},
  year={2009},
  address   = {New York, NY, USA},
  publisher={Springer Science \& Business Media}
}

@article{weh02,
  title={Analysis of upscaling absolute permeability},
  author={Wu, Xiao-Hui and Efendiev, Yalchin and Hou, Thomas Y},
  journal={Discrete and Continuous Dynamical Systems Series B},
  volume={2},
  number={2},
  pages={185--204},
  year={2002},
  publisher={AIMS PRESS}
}

@article{dur91,
  title={Numerical calculation of equivalent grid block permeability tensors for heterogeneous porous media},
  author={Durlofsky, Louis J},
  journal={Water Resources Research},
  volume={27},
  number={5},
  pages={699--708},
  year={1991},
  publisher={Wiley Online Library}
}

@article{ab05,
  title={A multiscale finite element method for numerical homogenization},
  author={Allaire, Gr{\'e}goire and Brizzi, Robert},
  journal={Multiscale Modeling \& Simulation},
  volume={4},
  number={3},
  pages={790--812},
  year={2005},
  publisher={SIAM}
}

@book{chung2023multiscale,
  title={Multiscale Model Reduction},
  author={Chung, Eric and Efendiev, Yalchin and Hou, Thomas Y},
  year={2023},
  address   = {Cham, Switzerland},  
  publisher={Springer}
}

@article{efendiev2013generalized,
  title={Generalized multiscale finite element methods (GMsFEM)},
  author={Efendiev, Yalchin and Galvis, Juan and Hou, Thomas Y},
  journal={Journal of Computational Physics},
  volume={251},
  pages={116--135},
  year={2013},
  publisher={Elsevier}
}

@article{ma2022novel,
  title={Novel design and analysis of generalized finite element methods based on locally optimal spectral approximations},
  author={Ma, Chupeng and Scheichl, Robert and Dodwell, Tim},
  journal={SIAM Journal on Numerical Analysis},
  volume={60},
  number={1},
  pages={244--273},
  year={2022},
  publisher={SIAM}
}

@article{henning2013oversampling,
  title={Oversampling for the multiscale finite element method},
  author={Henning, Patrick and Peterseim, Daniel},
  journal={Multiscale Modeling \& Simulation},
  volume={11},
  number={4},
  pages={1149--1175},
  year={2013},
  publisher={SIAM}
}

@article{kornhuber2018analysis,
  title={An analysis of a class of variational multiscale methods based on subspace decomposition},
  author={Kornhuber, Ralf and Peterseim, Daniel and Yserentant, Harry},
  journal={Mathematics of Computation},
  volume={87},
  number={314},
  pages={2765--2774},
  year={2018}
}

@book{book_owhadi2019operator,
  title={Operator-adapted wavelets, fast solvers, and numerical homogenization: from a game theoretic approach to numerical approximation and algorithm design},
  author={Owhadi, Houman and Scovel, Clint},
  volume={35},
  year={2019},
  address   = {Cambridge, UK},
  publisher={Cambridge University Press}
}

@article{babuska2011optimal,
  title={Optimal local approximation spaces for generalized finite element methods with application to multiscale problems},
  author={Babuska, Ivo and Lipton, Robert},
  journal={Multiscale Modeling \& Simulation},
  volume={9},
  number={1},
  pages={373--406},
  year={2011},
  publisher={SIAM}
}

@book{book_Peterseim,
author = {M\"{a}lqvist, Axel and Peterseim, Daniel},
title = {Numerical Homogenization by Localized Orthogonal Decomposition},
publisher = {Society for Industrial and Applied Mathematics},
year = {2020},
doi = {10.1137/1.9781611976458},
address = {Philadelphia, PA},
edition   = {},
URL = {https://epubs.siam.org/doi/abs/10.1137/1.9781611976458},
eprint = {https://epubs.siam.org/doi/pdf/10.1137/1.9781611976458}
}

@article{altmann2021numerical,
  title={Numerical homogenization beyond scale separation},
  author={Altmann, Robert and Henning, Patrick and Peterseim, Daniel},
  journal={Acta Numerica},
  volume={30},
  pages={1--86},
  year={2021},
  publisher={Cambridge University Press}
}

@book{grisvard2011elliptic,
  title={Elliptic problems in nonsmooth domains},
  author={Grisvard, Pierre},
  year={2011},
  address   = {Philadelphia, PA, USA},
  publisher={SIAM}
}

@article{landau1992theory,
  author  = {L. D. Landau and E. M. Lifshitz},
  title   = {On the theory of the dispersion of magetic permeability in ferromagnetic bodies},
  journal = {Phys. Z. Sowjetunion}, 
  volume  = {8},
  year    = {1935},
  pages   = {153-169}
}

@article{gilbert1955lagrangian,
  title="{A Lagrangian formulation of the gyromagnetic equation of the magnetization field}",
  author={Gilbert, Thomas L},
  journal={Phys. Rev.},
  volume={100},
  pages={1243-1255},
  year={1955}
}

@article{maalqvist2014localization,
  title={Localization of elliptic multiscale problems},
  author={M\"{a}lqvist, Axel and Peterseim, Daniel},
  journal={Mathematics of Computation},
  volume={83},
  number={290},
  pages={2583--2603},
  year={2014}
}

@article{li2013unconditional,
  title="{Unconditional convergence and optimal error estimates of a Galerkin-mixed FEM for incompressible miscible flow in porous media}",
  author={Li, Buyang and Sun, Weiwei},
  journal={SIAM Journal on Numerical Analysis},
  volume={51},
  number={4},
  pages={1959--1977},
  year={2013},
  publisher={SIAM}
}

@article{li2012error,
  title={Error analysis of linearized semi-implicit Galerkin finite element methods for nonlinear parabolic equations},
  author={Li, Buyang and Sun, Weiwei},
  journal={International Journal of Numerical Analysis and Modeling},
  volume={10},
  number={3},
  pages={622--633},
  year={2013},
  publisher={Global Science Press}
}

@article{cimrak2005error,
  title="{Error estimates for a semi-implicit numerical scheme solving the Landau-Lifshitz equation with an exchange field}",
  author={Cimr{\'a}k, Ivan},
  journal={IMA Journal of Numerical Analysis},
  volume={25},
  number={3},
  pages={611--634},
  year={2005},
  publisher={Oxford University Press}
}

@article{gao2014optimal,
  title="{Optimal error estimates of a linearized backward Euler FEM for the Landau-Lifshitz equation}",
  author={Gao, Huadong},
  journal={SIAM Journal on Numerical Analysis},
  volume={52},
  number={5},
  pages={2574--2593},
  year={2014},
  publisher={SIAM}
}

@article{liuGRPS,
author = {Liu, Xinliang and Zhang , Lei and Zhu , Shengxin},
title = {Generalized Rough Polyharmonic Splines for Multiscale PDEs with Rough Coefficients},
journal = {Numerical Mathematics: Theory, Methods and Applications},
year = {2021},
volume = {14},
number = {4},
pages = {862--892},
issn = {2079-7338},
doi = {https://doi.org/10.4208/nmtma.OA-2021-0100},
url = {http://global-sci.org/intro/article_detail/nmtma/19522.html},
}

@phdthesis{leitenmaier_Phdthesis,
  title={Analysis and numerical methods for multiscale problems in magnetization dynamics},
  author={Leitenmaier, Lena},
  year={2021},
  school={KTH Royal Institute of Technology}
}

@article{zhang_RPS,
  title={Polyharmonic homogenization, rough polyharmonic splines and sparse super-localization},
  author={Owhadi, Houman and Zhang, Lei and Berlyand, Leonid},
  journal={ESAIM: Mathematical Modelling and Numerical Analysis},
  volume={48},
  number={2},
  pages={517--552},
  year={2014},
  publisher={EDP Sciences}
}

@article{anrongbackward,
  title="{Analysis of backward Euler projection FEM for the Landau-Lifshitz equation}",
  author={An, Rong and Sun, Weiwei},
  journal={IMA Journal of Numerical Analysis},
  volume={42},
  number={3},
  pages={2336--2360},
  year={2022},
  publisher={Oxford University Press}
}

@article{santugini2007homogenization,
  title={Homogenization of ferromagnetic multilayers in the presence of surface energies},
  author={Santugini-Repiquet, K{\'e}vin},
  journal={ESAIM: Control, Optimisation and Calculus of Variations},
  volume={13},
  number={2},
  pages={305--330},
  year={2007},
  publisher={EDP Sciences}
}

@article{alouges2015homogenization,
  title={Homogenization of composite ferromagnetic materials},
  author={Alouges, Francois and Di Fratta, Giovanni},
  journal={Proceedings of the Royal Society A: Mathematical, Physical and Engineering Sciences},
  volume={471},
  number={2182},
  pages={20150365},
  year={2015},
  publisher={The Royal Society Publishing}
}

@article{choquet2018homogenization,
  title="{Homogenization of the Landau-Lifshitz-Gilbert equation in a contrasted composite medium}",
  author={Choquet, Catherine and Moumni, Mohammed and Tilioua, Mouhcine},
  journal={Discrete \& Continuous Dynamical Systems-S},
  volume={11},
  number={1},
  pages={35},
  year={2018},
  publisher={American Institute of Mathematical Sciences}
}

@article{alouges2021stochastic,
  title="{Stochastic homogenization of the Landau-Lifshitz-Gilbert equation}",
  author={Alouges, Francois and De Bouard, Anne and Merlet, Beno{\^\i}t and Nicolas, L{\'e}a},
  journal={Stochastics and Partial Differential Equations: Analysis and Computations},
  volume={9},
  pages={789--818},
  year={2021},
  publisher={Springer}
}

@article{UpscalingHMM,
author = {Leitenmaier, Lena and Runborg, Olof},
title = "{Upscaling Errors in Heterogeneous Multiscale Methods for the Landau--Lifshitz Equation}",
journal = {Multiscale Modeling \& Simulation},
volume = {20},
number = {1},
pages = {1-35},
year = {2022}
}

@article{leitenmaier2022heterogeneous,
  title="{Heterogeneous Multiscale Methods for the Landau-Lifshitz Equation}",
  author={Leitenmaier, Lena and Runborg, Olof},
  journal={Journal of Scientific Computing},
  volume={93},
  number={3},
  pages={76},
  year={2022},
  publisher={Springer}
}

@article{leitenmaier2022homogenization,
  title="{On homogenization of the Landau-Lifshitz equation with rapidly oscillating material coefficient}",
  author={Leitenmaier, Lena and Runborg, Olof},
  journal={Communications in Mathematical Sciences},
  volume={20},
  number={3},
  pages={653--694},
  year={2022},
  publisher={International Press of Boston}
}

@article{leitenmaier2023finite,
  title="{A finite element based Heterogeneous Multiscale Method for the Landau-Lifshitz equation}",
  author={Leitenmaier, Lena and Nazarov, Murtazo},
  journal={Journal of Computational Physics},
  volume={486},
  pages={112112},
  year={2023},
  publisher={Elsevier}
}

@article{chen2022multiscale,
  title="{On the Multiscale Landau-Lifshitz-Gilbert Equation: Two-Scale Convergence and Stability Analysis}",
  author={Chen, Jingrun and Du, Rui and Ma, Zetao and Sun, Zhiwei and Zhang, Lei},
  journal={Multiscale Modeling \& Simulation},
  volume={20},
  number={2},
  pages={835--856},
  year={2022},
  publisher={SIAM}
}

@article{owhadi2017gamblets,
  title={Gamblets for opening the complexity-bottleneck of implicit schemes for hyperbolic and parabolic ODEs/PDEs with rough coefficients},
  author={Owhadi, Houman and Zhang, Lei},
  journal={Journal of Computational Physics},
  volume={347},
  pages={99--128},
  year={2017},
  publisher={Elsevier}
}

@article{owhadi2017multigrid,
  title={Multigrid with rough coefficients and multiresolution operator decomposition from hierarchical information games},
  author={Owhadi, Houman},
  journal={Siam Review},
  volume={59},
  number={1},
  pages={99--149},
  year={2017},
  publisher={SIAM}
}

@article{discrete1990gronwall,
  title="{Finite-element approximation of the nonstationary Navier--Stokes problem. Part IV: error analysis for second-order time discretization}",
  author={Heywood, John G and Rannacher, Rolf},
  journal={SIAM Journal on Numerical Analysis},
  volume={27},
  number={2},
  pages={353--384},
  year={1990},
  publisher={SIAM}
}

@article{henning2022superconvergence,
  title="{Superconvergence of time invariants for the Gross--Pitaevskii equation}",
  author={Henning, Patrick and W{\"a}rnegard, Johan},
  journal={Mathematics of Computation},
  volume={91},
  number={334},
  pages={509--555},
  year={2022}
}

@article{henning2023optimal,
  title="{On Optimal Convergence Rates for Discrete Minimizers of the Gross--Pitaevskii Energy in Localized Orthogonal Decomposition Spaces}",
  author={Henning, Patrick and Persson, Anna},
  journal={Multiscale Modeling \& Simulation},
  volume={21},
  number={3},
  pages={993--1011},
  year={2023},
  publisher={SIAM}
}

@article{an2016LLCN,
  title="{Optimal error estimates of linearized Crank--Nicolson Galerkin method for Landau--Lifshitz equation}",
  author={An, Rong},
  journal={Journal of Scientific Computing},
  volume={69},
  number={1},
  pages={1--27},
  year={2016},
  publisher={Springer}
}

@book{maugeri2000elliptic,
  title={Elliptic and parabolic equations with discontinuous coefficients},
  author={Maugeri, Antonino and Palagachev, Dian K and SOFTOVA PALAGACHEVA, Lyoubomira and others},
  volume={109},
  year={2000},
  address   = {Berlin, Germany},
  publisher={WILEY-VCH Verlag GmbH \& Co.}
}

@article{bonizzoni2022super,
  title={Super-localized orthogonal decomposition for convection-dominated diffusion problems},
  author={Bonizzoni, Francesca and Freese, Philip and Peterseim, Daniel},
  journal={BIT Numerical Mathematics},
  volume={64},
  number={3},
  pages={33},
  year={2024},
  publisher={Springer}
}

@article{bonizzoni2024reduced,
  title={A reduced basis super-localized orthogonal decomposition for reaction-convection-diffusion problems},
  author={Bonizzoni, Francesca and Hauck, Moritz and Peterseim, Daniel},
  journal={Journal of Computational Physics},
  volume={499},
  pages={112698},
  year={2024},
  publisher={Elsevier}
}

@article{hauck2023super,
  title={Super-localization of elliptic multiscale problems},
  author={Hauck, Moritz and Peterseim, Daniel},
  journal={Mathematics of Computation},
  volume={92},
  number={341},
  pages={981--1003},
  year={2023}
}

@article{katzourakis2019numerical,
  title = {On the numerical approximation of p-biharmonic and $\infty$-biharmonic functions},
  author={Katzourakis, Nikos and Pryer, Tristan},
  journal={Numerical Methods for Partial Differential Equations},
  volume={35},
  number={1},
  pages={155--180},
  year={2019},
  publisher={Wiley Online Library}
}

\appendix

\section{Appendix}\label{sec:Appendix}
\subsection{Inequalities}

We recall the following key inequalities, which are used frequently in the analysis.

\begin{lemma}[Discrete Gronwall's inequality \cite{discrete1990gronwall}]\label{Gronwall}
Let $\tau$, $B$ and $a_k$, $b_k$, $c_k$, $\gamma_k$ be nonnegative numbers for all $k>0$
\begin{equation*}
    a_n +\tau  \sum \limits _{k=0} ^{n} b_k  \le \tau  \sum \limits _{k=0} ^{n} \gamma_k a_k + \tau  \sum \limits _{k=0} ^{n} c_k + B, ~n \geq 0 .
\end{equation*}
Suppose that $\tau \gamma_k < 1$ with $\sigma_k= (1-\tau \gamma_k)^{-1}$. Then,
\begin{equation*}
    a_n +\tau \sum \limits _{k=0} ^{n} b_k \le \exp(\tau \sum \limits _{k=0} ^{n} \gamma_k \sigma_k ) (\tau  \sum \limits _{k=0} ^{n} c_k + B),
\end{equation*}
holds for all $n \geq 0$.
\end{lemma}

\begin{lemma}[Agmon's inequality \cite{agmon2010lectures}]\label{lem:agmon_inequality}
Let \(\Omega \subset \mathbb{R}^d\) be a bounded Lipschitz domain.
\begin{itemize}
    \item[(a)] (Sobolev embedding) If \(m > d/2\), then there exists \(C = C(\Omega, d, m)\) such that
    \[
    \| u \|_{L^\infty(\Omega)} \le C \| u \|_{H^m(\Omega)}, \quad \forall u \in H^m(\Omega).
    \]
    \item[(b)] (Interpolation form) For an integer \(m > d/2\), there exists \(C = C(\Omega, d, m)\) such that
    \[
    \| u \|_{L^\infty(\Omega)} \le C \| u \|_{H^m(\Omega)}^{d/(2m)} \| u \|_{L^2(\Omega)}^{1 - d/(2m)}, \quad \forall u \in H^m(\Omega).
    \]
\end{itemize}
\end{lemma}

\begin{lemma}[Gagliardo–Nirenberg inequality \cite{nirenberg1966extended}]\label{lem:GN_inequality}
Let $\Omega \subset \mathbb{R}^d$ be a bounded Lipschitz domain.  
Let $m \in \mathbb{N}$, $j \in \mathbb{N}_0$ with $0 \le j < m$, 
and let $p, q, r \ge 1$ satisfy
\[
\frac{1}{p} = \frac{j}{d} + \theta\Bigl(\frac{1}{r} - \frac{m}{d}\Bigr) 
            + (1-\theta)\frac{1}{q},
\]
with
\[
\frac{j}{m} \le \theta \le 1,
\]
and exclude the case $m - j - \frac{d}{r} \in \mathbb{N}_0$ when $\theta = 1$. Then for every $u \in W^{m,r}(\Omega) \cap L^q(\Omega)$,
\[
\| D^j u \|_{L^{p}} 
\;\le\; C \, \| u \|_{W^{m,r}}^{\theta} \; \| u \|_{L^{q}}^{1-\theta},
\]
where the constant $C > 0$ depends on the parameters $j, m, d, q, r, \theta$, on the domain $\Omega$, but not on $u$.
\end{lemma}

\begin{lemma}[Calder\'on--Zygmund inequality \cite{maugeri2000elliptic}]\label{lem:C_Z_inequality}
Let \(\Omega \subset \mathbb{R}^d\) be a bounded and convex domain of class \(C^{2}\). Let $L$ be a uniformly elliptic operator with coefficients satisfying the uniform ellipticity condition. Then for any $1 < p < \infty$, there exists a constant \(C = C(d, p, \Omega) > 0\) such that
\[
\|u\|_{W^{2,p}(\Omega)} \leq C \|\Delta u\|_{L^p(\Omega)}.
\]
\end{lemma}

\subsection{Variants of Backward Euler Schemes}
\label{subsec: Some Other Backward Euler Schemes}

\subsubsection{Gao's Scheme}
 \cite{gao2014optimal} proposed a further simplification of the right-hand side treatment:
\begin{equation}
    \frac{1}{\tau}(\m_h^{n+1} - \m_h^{n},\v_h) - \alpha( \h_{\eff}^{n+1},\v_h) + ( \m_h^{n} \times \h_{\eff}^{n+1} , \v_h) = -\alpha( (\m^{n} \cdot \h_{\eff}^{n}) \cdot \m_h^{n} , \v_h).
    \label{GaoFEM}
\end{equation}
This scheme achieves optimal $L^2$ and $H^1$ error estimates without restrictive time-step conditions (i.e., $\tau = O(h^\alpha)$ is not required). The relaxation of the unit length constraint yields
\begin{equation}
    \|1 - |\m_h^n|^2 \|_{L^2} \le C_0 (\tau + h^{2}),
\end{equation}
demonstrating that the deviation from the unit sphere remains controlled.

\subsubsection{An's Scheme}
 \cite{anrongbackward} introduced a projection-based approach defined by
\begin{equation}
    \frac{1}{\tau}(\tilde{\m}_h^{n+1} - \m_h^{n},\v_h) - \alpha( \h_{\eff}^{n+1},\v_h) + ( \m_h^{n} \times \h_{\eff}^{n+1} , \v_h) = -\alpha ( (\m_h^{n} \cdot \Bar{\h}_{\eff}^{n+1}) \cdot \m_h^{n} , \v_h),
    \label{AnFEM}
\end{equation}
followed by a post-processing projection step $\m_h^{n+1} = \tilde{\m}_h^{n+1}/|\tilde{\m}_h^{n+1}|$. Here, $\Bar{\h}_{\eff}^{n+1} = \nabla \tilde{\m}_h^{n+1} \cdot \kappa \nabla \m_h^n$. This scheme achieves an optimal $L^2$ error estimate. Although the theoretical analysis requires $\tau = O(\eps_0 h)$ and initial data $\m_0 \in H^{r+1}$ with the polynomial degree of the finite element basis functions $r \geq 2$, numerical experiments indicate that these conditions can be relaxed in practice.

\end{document}